\documentclass[a4paper,10pt,bibtotoc]{scrartcl}
\usepackage{geometry}

\usepackage[british]{babel}
\usepackage[utf8]{inputenc}

\usepackage[colorlinks=true,urlcolor=black,linkcolor=black,citecolor=black]{hyperref}

\usepackage{amsmath}
\usepackage{amsthm}
\usepackage{amssymb}
\usepackage{mathtools}

\usepackage{enumitem}

\usepackage{caption}
\usepackage{graphicx}
\DeclareGraphicsExtensions{.png,.jpg}
\usepackage{subcaption}

\usepackage{comment}
\usepackage[normalem]{ulem}

\usepackage{tikz}
\usetikzlibrary{arrows}
\usetikzlibrary{positioning}
\usetikzlibrary{decorations}
\usetikzlibrary{calc}
\usetikzlibrary{mindmap}
\usetikzlibrary{shapes}
\usetikzlibrary{backgrounds}
\usetikzlibrary{intersections}

\usepackage{xcolor}

\newcommand*{\fullref}[1]{\hyperref[{#1}]{\ref*{#1} \nameref*{#1}}} %

\newcommand{\N}{\mathbb{N}}
\newcommand{\R}{\mathbb{R}}

\newcommand{\beq}{\begin{equation*}}		%
\newcommand{\eeq}{\end{equation*}}			%
\newcommand{\beqn}{\begin{equation}}		%
\newcommand{\eeqn}{\end{equation}}			%
\newcommand{\vn}[1]{\| #1 \|}
\newcommand{\st}{\quad{\text{s.t.}}\quad}

\DeclareMathOperator*{\supp}{supp}

\theoremstyle{definition}
\newtheorem{definition}{Definition}[section]
\newtheorem{example}[definition]{Example}

\theoremstyle{remark}
\newtheorem{remark}[definition]{Remark}

\theoremstyle{plain}
\newtheorem{proposition}[definition]{Proposition}
\newtheorem{lemma}[definition]{Lemma}
\newtheorem{theorem}[definition]{Theorem}
\newtheorem*{theorem*}{Theorem}
\newtheorem{corollary}[definition]{Corollary}

\newtheorem*{claim*}{Claim}

\definecolor{tud1a}{RGB}{93,133,195}
\definecolor{tud1b}{RGB}{0,90,169}

\definecolor{tud2a}{RGB}{0,156,218}
\definecolor{tud2b}{RGB}{0,131,204}
\definecolor{tud2c}{RGB}{0,104,157}
\definecolor{tud2d}{RGB}{0,78,115}

\definecolor{tud3a}{RGB}{80,182,149}
\definecolor{tud3b}{RGB}{0,157,129}
\definecolor{tud3c}{RGB}{0,136,119}
\definecolor{tud3d}{RGB}{0,113,94}

\definecolor{tud4a}{RGB}{175,204,80}
\definecolor{tud4b}{RGB}{153,192,0}
\definecolor{tud4c}{RGB}{127,171,22}
\definecolor{tud4d}{RGB}{106,139,55}

\definecolor{tud5b}{RGB}{201,212,0}
\definecolor{tud5d}{RGB}{153,166,4}

\definecolor{tud7b}{RGB}{210,135,0}

\definecolor{tud8a}{RGB}{238,122,52}
\definecolor{tud8b}{RGB}{236,101,0}
\definecolor{tud8c}{RGB}{204,76,3}

\definecolor{tud9a}{RGB}{233,80,62}

\definecolor{tud10a}{RGB}{201,48,142}
\definecolor{tud10b}{RGB}{166,0,132}
\definecolor{tud10c}{RGB}{149,17,105}
\definecolor{tud10d}{RGB}{115,32,84}

\definecolor{tud11a}{RGB}{128,69,151}
\definecolor{tud11b}{RGB}{114,16,133}
\definecolor{tud11c}{RGB}{97,28,115}
\definecolor{tud11d}{RGB}{76,34,106}

\newcommand{\cclc}[2]{\mathcal{L}^{CC}_#1(#2)}
\newcommand{\cccc}[2]{\mathcal{C}^{CC}_#1(#2)}

\numberwithin{equation}{section} 
\begin{document}

\thispagestyle{empty}

\ \bigskip

\begin{center}
   {\Large\bf Second Order Optimality Conditions and Improved Convergence Results for a Scholtes-type Regularization for a Continuous Reformulation of Cardinality Constrained Optimization Problems}
\end{center}

\bigskip

\begin{center}
{\large Max Bucher$^1$ and Alexandra Schwartz$^1$}
\end{center}

\medskip

{\large
\noindent \hskip4.2cm $^1$Technische Universit\"at Darmstadt 

\noindent \hskip4.2cm Graduate School Computational Engineering

\noindent \hskip4.2cm Dolivostra{\ss}e 15

\noindent \hskip4.2cm 64293 Darmstadt, Germany

\noindent \hskip4.2cm e-mail: \{bucher, schwartz\}@gsc.tu-darmstadt.de%
}
\bigskip

\begin{center}
	5th September 2017
\end{center}

\vfill

\begin{abstract}
	We consider nonlinear optimization problems with cardinality constraints.
	Based on a continuous reformulation we introduce second order necessary and sufficient optimality conditions.
	Under such a second order condition, we can guarantee local uniqueness of M-stationary points.
	Finally, we use this observation to provide extended local convergence theory for a Scholtes-type relaxation method for cardinality constrained optimization problems, which guarantees the existence and convergence of the iterates under suitable assumptions.
\end{abstract}

\bigskip %


\newpage

\setcounter{tocdepth}{1}
\section{Introduction}

In this article we consider cardinality constrained optimization problems of the form
\begin{equation}\label{eq:ccproblem}
	\min\limits_{x\in\R^n}\ f(x) \st x\in X,\ \vn{x}_0\leq \kappa,
\end{equation}
i.e. optimization problems that have, in addition to standard constraints $x \in X$, a bound $\kappa$ on the maximum number of nonzero components of $x$.

Problem \eqref{eq:ccproblem} can be used to model questions from a wide range of areas in science and industry.
Among its applications are the compressed sensing technique \cite{CW2008},
the subset selection problem in regression \cite{Miller2002},
support vector machines \cite{EWS2003},
cash management in automatic teller machines \cite{Galati2009},
lot sizing \cite{GK2013}
and portfolio optimization with constraints on the maximum number of assets \cite{B96}.

The cardinality constraint makes \eqref{eq:ccproblem} hard to solve: Despite its notation, $\vn{\cdot}_0$ is neither a norm nor a continuous mapping. 
Testing feasibility of \eqref{eq:ccproblem} is known to be NP-complete \cite{B96}.

For this reason in \cite{B96} problem \eqref{eq:ccproblem} is reformulated using binary auxiliary variables for the case where $X$ is polyhedral and $f$ a quadratic function,
which leads to the application of methods from discrete optimization \cite{B96,BS2009,LLRSS2012,MuSh2011}.
In \cite{LW2014} support recovery via nonconvex regularization is discussed.
For the special case $X=\R^n$ optimality conditions for \eqref{eq:ccproblem} from continuous optimization and algorithms are investigated in \cite{BE2013}.
In \cite{PXF2017} first and second order optimality conditions for \eqref{eq:ccproblem} are given.
These are formulated using the original cardinality constraint and suitable normal cones of the corresponding feasible set.

A recent approach is the reformulation of \eqref{eq:ccproblem} into a continuous optimization problem using orthogonality-type constraints.
This connection has been established in \cite{BKS2015} and \cite{FMPSW2013} and further studied in \cite{CKS2016,BBCS2017,GY2016}.
A similar reformulation for chance constrained optimization problems is discussed in \cite{Adam2016}, while in \cite{CWZ2016} penalization techniques for cardinality constraint optimization problems arising in the context of chance constraints are investigated.

While the reformulated optimization problem is continuous, the orthogonality-type constraints still pose difficulties which prevent a  direct application of methods from nonlinear optimization.
Most conditions that ensure that a local solution  satisfies first order optimality conditions, such as the well know linear independence constraint qualification, do not hold. 
The continuous reformulation bears a strong similarity to a mathematical program with complementarity constraints (MPCC).
This class of mathematical programs also violates most standard constraint qualifications.
For this reason a broad theory on MPCCs was developed, including custom constraint qualifications, stationary conditions and numerical methods.
For an overview of the subject of MPCCs (closely related to mathematical programs with equilibrium constraints (MPEC)) see \cite{LPR1996,OKZ1998}  and the references therein.
However, the continuous reformulation can not be embedded in the MPCC setting directly since it lacks the constraint $x\geq0$.
Additionally, it violates most of the MPCC-constraint qualifications, as argued in \cite{CKS2016}.
For this reason custom constraint qualifications and stationary conditions for the continuous reformulation were introduced \cite{BKS2015,CKS2016}.

In this article we present second order optimality conditions for the continuous reformulation, which make use of the aforementioned constraint qualifications and stationary conditions.
We prove both a necessary and a sufficient second order optimality condition for S-stationary points, which complement the first order optimality conditions in \cite{BKS2015,CKS2016}.
For M-stationary points, we prove a result for their uniqueness regarding the variable $x$ of the original problem \eqref{eq:ccproblem} also using a second order condition.
Compared to the second order optimality conditions from \cite{PXF2017}, the benefit of an analysis of the continuous reformulation are optimality conditions which are numerically exploitable with nonlinear programming methods as done for example in \cite{BBCS2017,BKS2015,FMPSW2013}.
Similar results for MPCCs and mathematical programs with vanishing constraints (MPVCs) can be found in \cite{GLY2013,HK2007,LPR1996,SS2000}.
For the classic results on nonlinear programs see for example \cite{NW2006}.

Moreover, we expand the convergence theory of a Scholtes-type regularization for the continuous reformulation.
It was shown in \cite{BBCS2017} that, under a Mangasarian-Fromowitz-type constraint qualification, KKT conditions of suitable regularized programs hold at a local minimum, and the limit of KKT points is an S-stationary point.
However, the question whether these regularized programs posses a solution is open. 
Using our results on second order optimality conditions, we illuminate the convergence properties of this regularization:
In a vicinity of a strict solution of \eqref{eq:ccproblem} the regularized programs have a solution. 
Moreover, the regularized programs have a solution close to a solution of the  continuous reformulation, if the cardinality constraint is active.
This leads to our main result regarding the Scholtes-regularization, which states that the limit point of a sequence of KKT points of the regularized programs is locally unique, provided a second order condition holds.

The remainder of the paper is structured as follows:
In the next section we introduce the continuous reformulation of the cardinality constrained problem \eqref{eq:ccproblem} and review some of the existing first order optimality conditions and custom constraint qualifications.
In Section \ref{sec:2nd-order} we derive necessary and sufficient second order optimality conditions for S-stationary points as well as a uniqueness result for M-stationary points using a second order condition.
And in Section \ref{sec:scholtes} we use these results to expand the convergence theory of a Scholtes-type regularization for the continuous reformulation.

To close this section let us introduce some notation used throughout this paper.
We use $\R_+ = [0,\infty)$ for the non negative real numbers.
For a given $x \in \R^n$ and $r > 0$ we denote the open ball with radius $r$ with respect to an arbitrary norm around $x$ by $B_r(x)$, its closure by $\overline{B_r(x)}$ and its boundary by $\partial B_r(x)$.
The unit vectors are denoted by $e_i \in \R^n$ and $e \in \R^n$ is the vector consisting of all ones.
For two vectors $x,y \in \R^n$ we denote the Hadamard product, i.e. the component-wise product, by $x \circ y \in \R^n$ and the connecting line between the vectors $x$ and $y$ by $[x,y]$.
A set of vectors $a_1,\ldots,a_m$ and $b_1,\ldots,b_p$ is called positively linearly dependent, if there exist multipliers $\lambda \in \R^m_+$ and $\mu \in \R^p$ such that $(\lambda, \mu) \neq 0$ and
\[
	\sum_{i=1}^m \lambda_i a_i + \sum_{i=1}^p \mu_i b_i = 0.
\]
For a given vector $x \in \R^n$ the support is given by $\supp(x) := \{i=1,\ldots,n \mid x_i \neq 0\}$ and the cardinality by $\|x\|_0 := |\supp(x)|$.
Throughout this paper, we assume all functions to be at least once continuously differentiable.
Whenever we need them to be twice continuously differentiable, this is stated explicitly.
For a function $f:\R^n \to \R$ the Jacobian $D f(x)$ is supposed to be a row vector whereas the gradient $\nabla f(x) = D f(x)^T$ is a column vector.
The Hessian matrix is denoted by $\nabla^2 f$.
Partial derivatives are indicated by subscripts i.e. $D_x f(x,y)$.

\section{Constraint Qualifications and First Order Optimality Conditions}\label{sec:cq-stat}

In this section, we introduce the continuous reformulation of the original cardinality constrained problem \eqref{eq:ccproblem} used in this paper and recall some first order optimality conditions and custom constraint qualifications previously introduced in \cite{BKS2015,CKS2016}.

Consider the cardinality constrained problem \eqref{eq:ccproblem}, where $f:\R^n \to \R$ is at least once continuously differentiable and the feasible set is given by
\[
	X \coloneqq \{x\in\R^n \mid g(x)\leq0,\ h(x)=0\} \subseteq \R^n
\]
with continuously differentiable functions $g:\R^n \to \R^m$ and $h: \R^n \to \R^p$.
To make the cardinality constraint meaningful, we assume $\kappa < n$.
Furthermore, we assume that
\[
	\mathcal{X} := \{ x \in X \mid \|x\|_0 \leq \kappa\} \neq \emptyset.
\]

In \cite{FMPSW2013,BKS2015} the following continuous reformulation of \eqref{eq:ccproblem} was introduced
\begin{equation}\label{eq:reformulation}
	\begin{aligned}
		\min\nolimits_{(x,y)\in\R^n\times\R^n}\ f(x) \st & g(x)\leq0, \quad h(x)=0,\\
		& 0\leq y \leq e, \quad e^Ty \geq n- \kappa, \\
		& x\circ y=0.
	\end{aligned}
\end{equation}
We denote the feasible set of \eqref{eq:reformulation} by $Z$.
Due to the orthogonality-type constraint $x \circ y = 0$ the auxiliary variable $y \in \R^n$ can be seen as a counter of the zero elements of $x$, of which there should be at least $n-\kappa$.

In \cite{BKS2015} and \cite{FMPSW2013} it was shown that $x^*$ is a global solution of \eqref{eq:ccproblem} if and only if there exists $y^*$ such that $(x^*,y^*)$ is a global solution of \eqref{eq:reformulation}.
Moreover, for each local solution $x^*$ of \eqref{eq:ccproblem} there exists a vector $y^*$ such that $(x^*,y^*)$ is a local solution of \eqref{eq:reformulation}.
However, not every local solution $(x^*,y^*)$ of the reformulation \eqref{eq:reformulation} necessarily corresponds to a local solution $x^*$ of the original problem \eqref{eq:ccproblem}.

To simplify the notation, we define the following index sets for a feasible point $(x^*,y^*)$ of the reformulation \eqref{eq:reformulation}:
\begin{align*}
	I_g(x^*) &:=\ \{i=1,\dots,m \mid g_i(x^*)=0\},\\
	I_0(x^*) &:=\ \{i=1,\dots,n \mid x^*_i=0\}.
\end{align*}
The set $\{1, \ldots,n\}$ is partitioned into
\begin{align*}
	I_{\pm0}(x^*,y^*) &:=\ \{i=1,\dots,n \mid x_i^*\not=0,\,y^*_i=0\},\\
	I_{00}(x^*,y^*)	&:=\ \{i=1,\dots,n \mid x_i^*=0,\,y^*_i=0\},\\
	I_{0+}(x^*,y^*)	&:=\ \{i=1,\dots,n \mid x_i^*=0,\,y^*_i\in(0,1)\},\\
	I_{01}(x^*,y^*)	&:=\ \{i=1,\dots,n \mid x_i^*=0,\,y^*_i=1\}.
\end{align*}
When the point of reference is obvious, we sometimes omit $(x^*,y^*)$ to keep the notation more compact.

Provided a constraint qualification holds, the KKT-conditions are a necessary first order optimality condition for a local minimum of a nonlinear optimization problem. 
However, for \eqref{eq:reformulation} standard constraint qualifications, like the linear independence constraint qualification (LICQ) or the Mangasarian-Fromowitz constraint qualification (MFCQ) (or even weaker ones), can not be expected to hold, see \cite{BKS2015} for details.
In \cite{BKS2015,CKS2016} alternative stationarity concepts for \eqref{eq:reformulation} have been introduced, which are first order optimality conditions under custom constraint qualifications.
We recall the definition of S- and M-stationarity next.
A comparison of further stationary concepts for the case $x\geq0$ can be found in \cite{CKS2016}. 

\begin{definition}\label{def:ccstationary}
	A feasible point $(x^*,y^*)\in Z$ of \eqref{eq:reformulation} is called
	\begin{enumerate}[label=(\alph*)]
		\item \emph{M-stationary} (M = Mordukhovich) if there exist multipliers $(\lambda^*,\mu^*,\gamma^*)\in \R^m\times\R^p\times\R^n$ such that
		\begin{align*}
			\nabla f(x^*)+\sum_{i=1}^m \lambda^*_i \nabla g_i(x^*) + \sum_{i=1}^p \mu^*_i \nabla h_i(x^*) + \sum_{i=1}^n \gamma^*_i e_i  & = 0,\\
			\lambda^*_i \geq 0,\quad \lambda^*_i \cdot g_i(x^*)& =0\quad \forall i=1,\dots,m,\\
			\gamma^*_i&=0\quad \forall i \in I_{\pm0}(x^*,y^*).
		\end{align*}

		\item \emph{S-stationary} (S = Strong) if is is M-stationary and $\gamma^*_i = 0 \quad \forall i \in I_{00}(x^*,y^*)$.

	\end{enumerate}
\end{definition}

Obviously M-stationarity is independent from $y^*$.
Thus we sometimes also say that a point $x^*$ is M-stationary and mean that $x^*$ is feasible for \eqref{eq:ccproblem}, i.e. there exists $y$ such that $(x^*,y)$ is feasible for \eqref{eq:reformulation}, and $x^*$ satisfies the definition of M-stationarity.

For the above stationary conditions to be necessary first order optimality conditions for a local minimum of \eqref{eq:reformulation}, we need suitable constraint qualifications to hold such as one of the following from \cite{CKS2016}.

\begin{definition}
	Let $(x^*,y^*)$ be feasible for \eqref{eq:reformulation}.
	We say that $(x^*,y^*)$ or $x^*$ satisfies
	\begin{enumerate}[label=(\alph*)]
		\item \emph{CC-LICQ (Cardinality Constrained - Linear Independence Constraint Qualification)} if and only if the gradients
		\begin{equation*}
			\nabla g_i(x^*)\ (i\in I_g(x^*)),\ \nabla h_i(x^*)\ (i=1,\dots,p),\ e_i\ (i\in I_0(x^*))
		\end{equation*}
		are linearly independent.

		\item \emph{CC-MFCQ (Cardinality Constrained - Mangasarian-Fromowitz Constraint Qualification)} if and only if the gradients
		\begin{equation*}
			\nabla g_i(x^*)\ (i\in I_g(x^*)),\quad\text{and}\quad \nabla h_i(x^*)\ (i=1,\dots,p),\ e_i\ (i\in I_0(x^*))
		\end{equation*}
		are positively linearly independent.

		\item \emph{CC-CPLD (Cardinality Constrained - Constant Positive Linear Dependence Constraint Qualification)} if for any subset $I_1\subseteq I_g(x^*)$, $I_2\subseteq \{1,\dots,p\}$ and $I_3\subseteq I_0(x^*)$ such that the gradients
		\begin{equation*}
			\nabla g_i(x)\ (i\in I_1),\quad\text{and}\quad \nabla h_i(x)\ (i\in I_2),\ e_i\ (i\in I_3)
		\end{equation*}
		are positively linearly dependent in $x=x^*$, they remain linearly dependent in a neighbourhood of $x^*$.
	\end{enumerate}
\end{definition}

The implications CC-LICQ $\Rightarrow$ CC-MFCQ $\Rightarrow$ CC-CPLD hold (see \cite{CKS2016}), which corresponds to the relations between the counterparts of the above constraint qualifications from the standard theory on nonlinear optimization.
Already under CC-CPLD, S-stationarity is a necessary first order optimality condition (cf. \cite[Theorem 4.2]{CKS2016}).
Here, the behaviour of the continuous reformulation \eqref{eq:reformulation} differs from the related class of MPCCs, where MPCC-LICQ is needed to guarantee S-stationarity of a local minimum.

\begin{remark}\label{rem:cc-mfcq-neighborhood}
	Consider a point $(x^*,y^*) \in Z$ satisfying CC-MFCQ.
	Due to the continuity of $g$ we know $I_g(x) \subseteq I_g(x^*)$ and $I_0(x) \subseteq I_0(x^*)$ for all $x$ sufficiently close to $x^*$.
	Thus, the continuity of $\nabla g$ and $\nabla h$ implies that there exists an $r > 0$ such that CC-MFCQ holds in all $(x,y) \in Z$ with $x \in B_r(x^*)$.
\end{remark}

As mentioned before, CC-LICQ guarantees that a local minimum $(x^*,y^*)$ of \eqref{eq:reformulation} is S-stationary and it is not hard to see that the corresponding multipliers are unique.
In case $x^*$ is even a local minimum of the original problem \eqref{eq:ccproblem}, a similar result can be obtained for all points $(x^*,y)$ feasible for \eqref{eq:reformulation}.

\begin{proposition}\label{prop:cclicq-s-multipl-unique}
	Let $x^*$ be a local minimum of \eqref{eq:ccproblem} satisfying CC-LICQ.
	Then every point $(x^*,y) \in Z$ is S-stationary.
	The corresponding multiplier $(\lambda^*,\mu^*,\gamma^*)\in\R^m\times\R^p\times\R^n$ is unique and independent from $y$.
	In case $\vn{x^*}_0<\kappa$ we additionally have $\gamma^*=0$.
\end{proposition}

\begin{proof}
	Since $x^*$ is a local minimum of \eqref{eq:ccproblem}, for all $y$ such that $(x^*,y) \in Z$ the point $(x^*,y)$ is a local minimum of \eqref{eq:reformulation} and thus due to CC-LICQ an S-stationary point, see \cite[Theorem 4.2]{CKS2016}.
 	Hence there exist S-stationary multipliers $(\lambda^*,\mu^*,\gamma^*)$ with $\lambda^*_i = 0$ for all $i \notin I_g(x^*)$ and $\gamma^*_i = 0$ for all $i \notin I_0(x^*) $ and
	\begin{equation*}
		\nabla f(x^*)+\sum\limits_{i \in I_g(x^*)}\lambda_i^*\nabla g_i(x^*)+\sum\limits_{i=1}^p\mu_i^*\nabla  h_i(x^*)+\sum\limits_{i \in I_0(x^*)}\gamma_i^*e_i = 0
	\end{equation*}
	Due to CC-LICQ this equation
	has at most one solution $(\lambda^*_{I_g},\mu^*,\gamma^*_{I_0})$ and thus the multiplier $(\lambda^*,\mu^*,\gamma^*)$ is unique and independent from $y$.
	
	Let $\vn{x^*}_0<\kappa$.
	It remains to show that in this case $\gamma^* = 0$.
	For all $i \notin I_0(x^*)$ this is guaranteed by the definition of S-stationarity.
	For every  $j \in I_0(x^*)$ we can define
	\begin{equation}
		y_i = \begin{cases}
			0 & \text{if } i \in \supp(x^*) \cup \{j\}, \\
			1 & \text{else.}
		\end{cases} 
	\end{equation}
	Because $|I_0(x^*)|>n-\kappa$ the point $(x^*,y)$ is feasible for \eqref{eq:reformulation} and thus a local minimum and S-stationary point of \eqref{eq:reformulation}.
	The S-stationarity conditions then imply $\gamma^*_j = 0$.
	Since the multiplier $(\lambda^*,\mu^*,\gamma^*)$ is unique and independent from $y$ and the same argument holds for all $j \in I_0(x^*)$, we have shown $\gamma^* = 0$. \qed
\end{proof}

In the recent paper \cite{PXF2017} the authors also derive first order optimality conditions for \eqref{eq:ccproblem} based on Fr\'{e}chet, limiting and Clarke normal cones, which are called B-, M-, and C-KKT points.
Instead of the constraint qualifications previously introduced here, they use conditions called R-LICQ and R-MFCQ, which however can be shown to be equivalent to CC-LICQ and CC-MFCQ.
In \cite[Theorem 3.2]{PXF2017} it is shown that a local minimum of \eqref{eq:ccproblem} is a B-KKT point under R-LICQ.
Closer inspection of the involved definitions reveals that this is equivalent to Proposition \ref{prop:cclicq-s-multipl-unique}.
To ensure that M-KKT points, which are related to S-stationary points, or C-KKT points, which are equivalent to M-stationary points, are necessary optimality conditions at local minima of \eqref{eq:ccproblem}, R-MFCQ is required to hold in \cite[Theorem 3.4, Theorem 3.5]{PXF2017}.
Using the continuous reformulation instead of the normal cone approach, one can show that instead of R-MFCQ weaker conditions such as CC-CPLD are enough to guarantee M- and S-stationarity of local minima, see \cite{BKS2015,CKS2016}.

\section{Second Order Optimality Conditions}\label{sec:2nd-order}

In this section we derive second order optimality conditions for \eqref{eq:reformulation}.
We begin with a second order necessary optimality condition for local solutions of \eqref{eq:reformulation} which holds under CC-LICQ. 
For S-stationary points we then derive a second order sufficient optimality condition for a certain type of strict local minima.
Moreover, we show that M-stationary points are locally unique, provided CC-CPLD and a second order condition hold.

To formulate these optimality conditions, we need to define the linearisation cone and the critical cone first.
We use the \emph{CC-linearisation cone}, which was introduced in \cite{CKS2016} and used there to derive the custom constraint qualifications and first order optimality conditions for \eqref{eq:reformulation}, see Section \ref{sec:cq-stat}.

\begin{definition}%
	Let $(x^*,y^*)\in Z$ be a feasible point of \eqref{eq:reformulation}.
	The \emph{CC-linearisation cone} is defined by
	\begin{equation}\label{eq:cclc}
		\begin{array}{rll}
			\cclc{Z}{x^*,y^*} = \{(d_x,d_y)\in\R^n\times\R^n \mid & \nabla g_i(x^*)^Td_x\leq0 & \forall i\in I_g,\\
			& \nabla h_i(x^*)^Td_x = 0 &\forall i=1,\dots,p,\\
			& e^Td_y \geq 0 & \text{if } e^Ty^*=n-\kappa,\\
			& e_i^Td_y=0 & \forall i\in I_{\pm0},\\
			& e_i^Td_y\geq0 & \forall i\in I_{00},\\
			& e_i^Td_y\leq0 & \forall i\in I_{01},\\
			& e_i^Td_x=0 & \forall i\in I_{01} \cup I_{0+},\\
			& \left(e_i^Td_x\right)\left(e_i^Td_y\right) = 0 & \forall i\in I_{00}\}.
		\end{array}
	\end{equation}
\end{definition}

Later, we are mostly interested in the directions $d_x$ only.
It is straightforward to see that
\begin{equation*}
	\begin{array}{rcl}
		\{d_x \in \R^n & \mid & \exists d_y \in \R^n \ : \  (d_x,d_y) \in \cclc{Z}{x^*,y^*}\} \\
		= \{d_x \in \R^n & \mid & \nabla g_i(x^*)^Td_x \leq0 \qquad \forall i\in I_g(x^*),\\
		&& \nabla h_i(x^*)^Td_x = 0 \qquad \forall i=1,\dots,p,\\
		&& e_i^T d_x = 0 \qquad\qquad\quad \forall i\in I_{01}(x^*,y^*) \cup I_{0+}(x^*,y^*) \}.
	\end{array}
\end{equation*}
In case $\|x^*\|_0 < \kappa$, this set still depends on the chosen value of $y^*$.
Thus for a given $x^*$ we also consider the union over all $y^*$ such that $(x^*,y^*) \in Z$:
\begin{equation*}
	\begin{array}{rcl}
		\mathcal{L}_{\mathcal X}(x^*) := \{d_x \in \R^n & \mid & \exists y^* \in \R^n, d_y \in \R^n \ : (x^*,y^*)\in Z, \  (d_x,d_y) \in \cclc{Z}{x^*,y^*} \} \\
		= \{d_x \in \R^n & \mid & \nabla g_i(x^*)^Td_x \leq0 \qquad \forall i\in I_g(x^*),\\
		&& \nabla h_i(x^*)^Td_x = 0 \qquad \forall i=1,\dots,p,\\
		&& |\{i \in I_0(x^*) \mid (d_x)_i = 0\}|  \geq n - \kappa \}.
	\end{array}
\end{equation*}
In a certain sense $\mathcal{L}_{\mathcal X}(x^*)$ can be seen as a linearisation cone for the original feasible set $\mathcal X$.
Note however that $\mathcal{L}_{\mathcal X}(x^*)$ in nonconvex in case $\|x^*\|_0 < \kappa$.

The \emph{CC-critical cone}, see also \cite{HK2007,LPR1996,NW2006} for related constructions, is then the set of all potential feasible descent directions.
\begin{definition}\label{def:cc-critical-cone}
	Let $(x^*,y^*)\in Z$.
	The \emph{CC-critical cone} of $Z$ at $(x^*,y^*)$ is defined by
	\begin{equation*}
		\cccc{Z}{x^*,y^*}:=\cclc{Z}{x^*,y^*}\ \cap\ \left\{(d_x,d_y)\in\R^n\times\R^n \mid\nabla f(x^*)^Td_x\leq0\right\}.
	\end{equation*}
	A vector $d\in\cccc{Z}{x^*,y^*}$ is called \emph{critical direction }(\emph{at $(x^*,y^*)$}).
\end{definition}

Analogously we define
\begin{eqnarray*}
	\mathcal{C}_\mathcal{X}(x^*) & :=& \{d_x \in \mathcal{L}_{\mathcal X}(x^*) \mid \nabla f(x^*)^Td_x \leq 0 \}.
\end{eqnarray*}

If $(x^*,y^*)$ is an S-stationary point of \eqref{eq:reformulation}, we can give a description of $\cccc{Z}{x^*,y^*}$ that does not use the gradient of the objective function but instead the multipliers of $(x^*,y^*)$.
For some multipliers $\lambda^*\in\R^m_+$ corresponding to the inequality constraints $g(x^*)\leq0$ we define the index sets
\begin{eqnarray*}
		I_{g+}(x^*,\lambda^*) & \coloneqq & \{i\in I_g(x^*) \mid \lambda^*_i>0\},\\
		I_{g0}(x^*,\lambda^*) & \coloneqq & \{i\in I_g(x^*) \mid \lambda^*_i=0\}.
\end{eqnarray*}
With these index sets, the following proposition gives a characterization of the CC-critical cone for at S-stationary point.

\begin{proposition}\label{prop:cc-cone-multiplier-description}
	Let $(x^*,y^*)$ be an S-stationary point of \eqref{eq:reformulation} with multipliers $(\lambda^*,\mu^*,\gamma^*)$.
	Then we have
	\begin{equation*} 
		\cccc{Z}{x^*,y^*} =  \{(d_x,d_y)^T\in\cclc{Z}{x^*,y^*} \mid \nabla g_i(x^*)^Td_x=0\   \forall i\in I_{g+}(x^*,\lambda^*)\}.
	\end{equation*}
\end{proposition}

\begin{proof}
	Let $(d_x,d_y)\in\cclc{Z}{x^*,y^*}$ be arbitrary.
	It suffices to show the equivalence
	\begin{equation*}
		\nabla f(x^*)^Td_x\leq0\quad\Leftrightarrow\quad \nabla g_i(x^*)d_x=0\quad\forall i\in I_{g+}(x^*,\lambda^*).
	\end{equation*}
	Since $(x^*,y^*)$ is S-stationary with  multipliers $(\lambda^*,\mu^*,\gamma^*)$ we know $\lambda^* \geq 0$ and 
	\begin{equation*}
			\nabla f(x^*)=-\sum\limits_{i\in I_g}\lambda^*_i\nabla g_i(x^*)-\sum\limits_{i=1}^p\mu_i^*\nabla h_i(x^*) -\sum\limits_{i\in I_{0+}\cup I_{01}}\gamma_i^*e_i.
	\end{equation*}
	Taking into account $(d_x,d_y)\in\cclc{Z}{x^*,y^*}$, we obtain
	\begin{eqnarray*}
		&& \nabla f(x^*)^T d_x \leq 0\\
		 & \quad \Leftrightarrow \quad & -\sum_{i \in I_g} \lambda^*_i \nabla g_i(x^*)^T d_x- \sum_{i=1}^p \mu_i^* \nabla h_i(x^*)^T d_x - \sum_{i \in I_{0+} \cup I_{01}} \gamma_i^ * e_i^T d_x \leq 0 \\
		& \quad \Leftrightarrow \quad & -\sum_{i \in I_g} \lambda^*_i \nabla g_i(x^*)^T d_x \leq 0\\
		& \quad \Leftrightarrow \quad & \nabla g_i(x^*)^T d_x  =0 \quad \forall i \in I_{g+}(x^*,\lambda^*),
	\end{eqnarray*}
	since $\lambda_i^*\geq0$ for all $i=1,\dots,m$. \qed
\end{proof}

Note that the alternative representation from Proposition \ref{prop:cc-cone-multiplier-description} does not necessarily hold for M-stationary points.
The reason is that for an M-stationary point $(x^*,y^*)$ with multipliers $(\lambda^*,\mu^*,\gamma^*)$ and a vector $(d_x,d_y)\in \cclc{Z}{x^*,y^*}$ the equation $\gamma_i^*e_i^Td_x=0$ does not necessary hold for $i\in I_{00}(x^*,y^*)$.

We now proceed to derive second order necessary and sufficient optimality conditions for S-stationary points and uniqueness of M-stationary points under a second order condition.

\subsection{Second Order Necessary Optimality Condition}

Our next goal is to derive a second order necessary optimality condition for local minima of the continuous reformulation \eqref{eq:reformulation}.
Its proof is similar to the approach known from classical nonlinear optimization, see for example \cite{NW2006}, and from mathematical programs with vanishing constraints (MPVCs), see \cite{HK2007}. 

To be able to prove this result, we need the following auxiliary lemma first.
Note that this result requires linear independence constraint qualification to hold, since it is based on an implicit function theorem.

\newcommand{\nccurve}{\xi}
\begin{lemma}\label{lem:ssonc-curve}
	Let $g,h$ be twice continuously differentiable, $(x^*,y^*)$ be an S-stationary point of \eqref{eq:reformulation} with multipliers $(\lambda^*, \mu^*,\gamma^*)$ satisfying CC-LICQ, and $d\coloneqq(d_x,d_y)\in\cccc{Z}{x^*,y^*}$.
	Then there exists an $\varepsilon>0$ and a twice continuously differentiable curve $\nccurve: (-\varepsilon,\varepsilon) \to \R^n$ with $\nccurve(0)=x^*$, $\nccurve'(0)=d_x$ and
	\begin{align}
		(\nccurve(t),y^*)\in Z\quad& \forall t\in[0,\varepsilon),\label{curve:properties0}\\
		g_i(\nccurve(t))=0\quad&\forall i\in I_{g+}(x^*,\lambda^*),\ \forall t\in[0,\varepsilon),&\label{curve:properties1}\\
		\nccurve_i(t)=0\quad&\forall I_{0+}(x^*,y^*)\cup I_{01}(x^*,y^*),\ \forall t\in[0,\varepsilon).\label{curve:properties2}
	\end{align}	
\end{lemma}

\begin{proof}
	For the given vector $d = (d_x,d_y) \in \cccc{Z}{x^*,y^*}$ split the index set $I_{g,0}(x^*)$ into
	\begin{align*} 
		I^=_{g,0}(x^*,\lambda^*,d)\coloneqq\left\{i\in I_{g0}(x^*,\lambda^*) \mid \nabla g_i(x^*)^Td_x=0\right\},\\
		I^<_{g,=}(x^*,\lambda^*,d)\coloneqq\left\{i\in I_{g0}(x^*,\lambda^*) \mid \nabla g_i(x^*)^Td_x<0\right\}.
	\end{align*}
	To keep the notation compact, we also use the abbreviations
	\begin{align*}
		\mathcal{A}(x^*)\coloneqq\ & I_{g+}(x^*,\lambda^*)\cup I_{g0}^=(x^*,\lambda^*,d),\\
		\mathcal{B}(x^*,y^*)\coloneqq\ & I_{0+}(x^*,y^*)\cup I_{01}(x^*,y^*),\\
		M\coloneqq\ & |\mathcal{A}(x^*)|+p+|\mathcal{B}(x^*,y^*)|
	\end{align*}
	and define $q:\R^n\rightarrow\R^M$,
	\begin{equation*} 
		x\mapsto q(x)\coloneqq
		\begin{pmatrix*}[c]
			g_{\mathcal{A}(x^*)}(x)\\
			h(x)\\
			x_{\mathcal{B}(x^*,y^*)}
		\end{pmatrix*}.
	\end{equation*}
	Since $g$ and $h$ are twice continuously differentiable, so is $q$.
	Using the function $q$ we now define $F:\R^M\times\R\rightarrow\R^M$, $(v,t)\mapsto F(v,t)$, by
	\begin{equation*} 
		F(v,t)\coloneqq q(x^*+t\cdot d_x+Dq(x^*)^Tv).
	\end{equation*}
	Using the chain rule we can calculate
	\begin{equation*} 
		D_v F(v,t)=Dq(x^*+td_x+Dq(x^*)^Tv)Dq(x^*)^T.
	\end{equation*}
	Consider the point $(v^*,t^*)\coloneqq(0,0)$.
	We have $F(v^*,t^*)=0$ and 
	\begin{equation*} 
		D_v F(v^*,t^*)=Dq(x^*)Dq(x^*)^T.
	\end{equation*}
	Since CC-LICQ holds in $(x^*,y^*)$ the matrix $Dq(x^*)\in\R^{M\times n}$ has full row rank and therefore the matrix $D_vF(v^*,t^*)$ is regular.
	The function $F$ is twice continuously differentiable.
	Thus the implicit function theorem (see \cite[Theorem A.2]{NW2006}) provides the existence of an $\varepsilon>0$ and twice continuously differentiable curve $\nu: (-\varepsilon,\varepsilon) \to \R^M$ with the properties $\nu(0)=0$ and for all $t\in(-\varepsilon,\varepsilon)$
	\begin{equation}\label{hilfskurve}
		\begin{aligned}
			& F(\nu(t),t)=0, \qquad\det D_v(\nu(t),t)\not=0,\\
			& \nu'(t)=-\left(D_vF(\nu(t),t)\right)^{-1}D_tF(\nu(t),t).
		\end{aligned}
	\end{equation}
	We have $D_t F(v,t)=Dq(x^*+td_x+Dq(x^*)v)\cdot d_x$ and, using \eqref{hilfskurve}, therefore
	\begin{align}
		\nu'(0) & = -\left(D_vF(\nu(0),0)\right)^{-1}\cdot Dq(x^*)\cdot d_x\notag\\
		& = -\left(D_vF(\nu(0),0)\right)^{-1}\cdot \underbrace{{
			\begin{pmatrix*}[l]
				\nabla g_{\mathcal{A}(x^*)}(x^*)^T\\
				\nabla h(x^*)^T\\
				e_{\mathcal{B}(x^*,y^*)}^T
			\end{pmatrix*}}
			\cdot d_x}_{=0}\ = 0\label{hilfskurve:ableitung0},
	\end{align}
	where we used the $(d_x,d_y)\in\cccc{Z}{x^*,y^*}$, see also Proposition \ref{prop:cc-cone-multiplier-description}.
	Lastly we define $\nccurve:(-\varepsilon,\varepsilon)\rightarrow\R^n$ by $t\mapsto \nccurve(t)\coloneqq x^*+t\cdot d_x+Dq(x^*)^T\nu(t)$.
	Because the involved functions are twice continuously differentiable, so is the function $\nccurve$  and we have $\nccurve'(t)=d_x+Dq(x^*)^T\nu'(t)$ for all $t\in(-\varepsilon,\varepsilon)$.
	
	From \eqref{hilfskurve} and \eqref{hilfskurve:ableitung0} it follows that $\nccurve(0)=x^*$ and
	\begin{align*}
		\nccurve'(0)=d_x+Dq(x^*)^T\nu'(0)=d_x.
	\end{align*}

	Next, we show the feasibility of the vector $(\nccurve(t),y^*)$. 
	To this end, note that for all $j$ and all $t\in(-\varepsilon,\varepsilon)$ we have $q_j(\nccurve(t)) = F_j(v(t),t) = 0$ and thus
	\begin{align}
		g_i(\nccurve(t))=0 &\quad \forall i\in I_{g+}(x^*,\lambda^*)\cup I_{g0}^=(x^*,\lambda^*,d),\label{nclem:feasibility1}\\
		h_i(\nccurve(t))=0 &\quad \forall i\in\{1,\dots,p\},\label{nclem:feasibility2}\\
		\nccurve_i(t)=0 &\quad \forall i\in I_{0+}(x^*,y^*)\cup I_{01}(x^*,y^*)\label{nclem:feasibility3}.
	\end{align}
	Since $(x^*,y^*)\in Z$ the constraints $0\leq y_i^* \leq 1$ and $\sum_{i=1}^n y_i^*\geq n-\kappa$ hold.
	Because of $y_i^*=0$ for all $i\in I_{00}(x^*,y^*)\cup I_{\pm0}(x^*,y^*)$ and \eqref{nclem:feasibility3}, the complementarity constraint $\nccurve_i(t)\cdot y_i^*=0$ holds for all $i\in\{1,\dots,n\}$ and all $t\in(-\varepsilon,\varepsilon)$.
	
	Let $i\in I_{g0}^<(x^*,\lambda^*,d)$.
	We have
	\begin{equation*}
		\frac{d}{dt} (g_i\circ\nccurve)(0) = \nabla g_i(\nccurve(0))^T\nccurve'(0) = \nabla g_i(x^*)^Td_x<0 
	\end{equation*}
	and therefore $g_i(\nccurve(t))<0$ for all $t>0$ sufficiently small.
	
	Since the sets $I_{g0}^<(x^*,\lambda^*,d)$, $I_{g0}^=(x^*,\lambda^*,d)$ and $I_{g+}(x^*,\lambda^*)$ form a partition of $I_g(x^*)$ we have thus shown that the constraint $g_i(\nccurve(t))\leq 0$ holds for all $i\in I_g(x^*)$ and all $t\in [0,\varepsilon)$ with $\varepsilon > 0$ sufficiently small.
	
	Let $i\in \{1,\dots,m\}\setminus I_g(x^*)$ and thus $g_i(x^*)<0$.
	Due to the continuity of $g_i$ and $\nccurve$ the inequality $g_i(\nccurve(t))<0$ then still holds for all $t > 0$ sufficiently small.
	Consequently we have verified $g_i(\nccurve(t))\leq0$ for all $i\in\{1,\dots,m\}$ and, together with \eqref{nclem:feasibility2}, $\nccurve(t)\in X$ for all $t\in[0,\varepsilon)$, if $\varepsilon > 0$ is chosen sufficiently small.
	
	Altogether we have proven $(\nccurve(t),y^*)\in Z$ for all $t\in [0,\varepsilon)$.
	The properties \eqref{curve:properties1} and \eqref{curve:properties2} follow from \eqref{nclem:feasibility1} and \eqref{nclem:feasibility3}. \qed
\end{proof}

Using this result, we can now proceed with the second order necessary condition for the continuous reformulation \eqref{eq:reformulation}.
For its proof we follow an idea in \cite{HK2007}.

\begin{theorem}[Second Order Necessary Optimality Condition]\label{thm:cc-ssonc}
	Let $f,g,h$ be twice continuously differentiable, $(x^*,y^*)$ be a local minimum of \eqref{eq:reformulation} satisfying CC-LICQ, and $(\lambda^*,\mu^*,\gamma^*)$ be the unique S-stationary multipliers for $(x^*,y^*)$.
	Then
	\begin{equation*}
		d_x^T\left(\nabla^2f(x^*)+\sum\limits_{i=1}^m\lambda^*_i\nabla^2 g_i(x^*)+\sum\limits_{i=1}^p\mu^*_i\nabla^2h_i(x^*)\right)d_x\geq 0
	\end{equation*}
	for all $(d_x,d_y)^T\in \cccc{Z}{x^*,y^*}$.
\end{theorem}

\begin{proof}
	Since the CC-LICQ holds in the local minimum $(x^*,y^*)$, this point is also an S-stationary point of \eqref{eq:reformulation} with unique multipliers $(\lambda^*,\mu^*,\gamma^*)$.

	Let $d=(d_x,d_y)^T\in \cccc{Z}{x^*,y^*}$ be arbitrary.
	Due to Lemma \ref{lem:ssonc-curve} there exists an $\varepsilon>0$ and twice continuously differentiable curve $\nccurve: (-\varepsilon,\varepsilon) \to \R^n$ with the properties $\nccurve(0)=x^*$, $\nccurve'(0)=d_x$ and
	\begin{align}
		(\nccurve(t),y^*)\in Z\quad& \forall t\in[0,\varepsilon),\label{curve:properties0b}\\
		g_i(\nccurve(t))=0\quad&\forall i\in I_{g+}(x^*,\lambda^*),\ \forall t\in[0,\varepsilon),&\label{curve:properties1b}\\
		\nccurve_i(t)=0\quad&\forall I_{0+}(x^*,y^*) \cup I_{01}(x^*,y^*),\ \forall t\in[0,\varepsilon).\label{curve:properties2b}
	\end{align}	
	We define $\ell:\R^n\times\R^m\times\R^p\times\R^n\rightarrow\R$,
	\begin{equation*} 
		(x,\lambda,\mu,\gamma)\mapsto \ell(x,\lambda,\mu,\gamma)\coloneqq f(x)+\sum\limits_{i=1}^m\lambda_i g_i(x)+\sum\limits_{i=1}^p\mu_i h_i(x)+\sum\limits_{i=1}^n\gamma_i x_i.
	\end{equation*}
	Due to our assumptions on $f$, $g$ and $h$, this function is also twice continuously differentiable and since $(x^*,y^*)$ is S-stationary, we know
	\begin{equation}\label{thm:ssonc-gradientl}
		\nabla_x\ell(x^*,\lambda^*,\mu^*,\gamma^*)=0.
	\end{equation}

	Define the function $\varphi:(-\varepsilon,\varepsilon)\rightarrow\R$ by $t\mapsto \varphi(t)\coloneqq \ell(\nccurve(t),\lambda^*,\mu^*,\gamma^*)$.
	Combining \eqref{curve:properties1b} and \eqref{curve:properties2b} with the fact that $(x^*,y^*)$ with the multipliers $(\lambda^*, \mu^*, \gamma^*)$ is S-stationary, we obtain for all $t\in[0,\varepsilon)$
	\begin{equation}\label{thm:ssonc-phifzusammenhang} 
		\varphi(t)=f(\nccurve(t))+\sum\limits_{i=1}^m\lambda^*_i g_i(\nccurve(t))+\sum\limits_{i=1}^p\mu^*_i h_i(\nccurve(t))+\sum\limits_{i=1}^n\gamma^*_i \nccurve_i(t)=f(\nccurve(t)). 
	\end{equation}
	The function $\varphi$ is twice continuously differentiable with
	\begin{align*}
		\varphi'(t) & = \nabla_x\ell(\nccurve(t),\lambda^*,\mu^*,\gamma^*)^T\cdot\nccurve'(t),\\
		\varphi''(t) & = \nccurve'(t)^T\nabla_{xx}\ell(\nccurve(t),\lambda^*,\mu^*,\gamma^*)\nccurve'(t)+\nabla_x\ell(\nccurve(t),\lambda^*,\mu^*,\gamma^*)^T\cdot\nccurve''(t)
	\end{align*}
	for all $t\in[0,\varepsilon)$.
	Using \eqref{curve:properties0b}, \eqref{thm:ssonc-gradientl} and \eqref{thm:ssonc-phifzusammenhang}, we obtain $\varphi(0) =f(x^*)$, $\varphi'(0) =0$ and
	\begin{eqnarray*} 
		\varphi''(0)  &=& d_x^T \nabla_{xx}\ell(x^*,\lambda^*,\mu^*,\gamma^*) d_x \\
		& =& d_x^T\left(\nabla_x^2f(x^*)+\sum\limits_{i=1}^m\lambda^*_i\nabla_x^2 g_i(x^*)+\sum\limits_{i=1}^p\mu^*_i\nabla_x^2h_i(x^*)\right)d_x.
	\end{eqnarray*}

	To conclude the proof assume that $\varphi''(0) = d_x^T \nabla_{xx}\ell(x^*,\lambda^*,\mu^*,\gamma^*) d_x<0$.
	Because $\varphi$ is twice continuously differentiable, the inequality $\varphi''(0)<0$ implies $\varphi''(t)<0$ for all $|t|$ sufficiently small.
	For $t>0$ sufficiently small Taylor's theorem provides the existence of a $\theta_t\in[0,t)$ such that
	\begin{equation*}
		\varphi(t) = \varphi(0) + t\cdot\underbrace{\varphi'(0)}_{=0}+\frac{t^2}{2}\underbrace{\varphi''(\theta_t)}_{<0}.
	\end{equation*}
	Thus for $t > 0$ sufficiently small  we obtain $\varphi(t)<\varphi(0)$ (note $\theta_t\rightarrow0$ for $t\rightarrow0$).
	Altogether we can argue that
	\begin{equation*}
		f(\nccurve(t))=\varphi(t)<\varphi(0)=f(x^*)
	\end{equation*}
	for $t > 0$ sufficiently small.
	Since $(\nccurve(t),y^*)$ is feasible for \eqref{eq:reformulation} for all $t\in[0,\varepsilon)$ and $(\nccurve(t),y^*) \to (x^*,y^*)$ for $t \downarrow 0$, this is a contradiction to $(x^*,y^*)$ being a local minimum of \eqref{eq:reformulation}. \qed
\end{proof}

If $x^*$ is a local minimum of \eqref{eq:ccproblem} satisfying CC-LICQ, we know that every feasible point $(x^*,y) \in Z$ is a local minimum and thus S-stationary point of \eqref{eq:reformulation}.
By Proposition \ref{prop:cclicq-s-multipl-unique} all S-stationary points $(x^*,y)$ share a unique multiplier $(\lambda^*,\mu^*,\gamma^*)$.
Thus, as a corollary we immediately recover the second order necessary sufficient condition from \cite[Theorem 4.1]{PXF2017}:

\begin{corollary}\label{cor:cc-ssonc}
	Let $f,g,h$ be twice continuously differentiable, $x^*$ be a local minimum of \eqref{eq:ccproblem} satisfying CC-LICQ, and $(\lambda^*,\mu^*,\gamma^*)$ be the unique S-stationary multiplier for all $(x^*,y) \in Z$.
	Then
	\begin{equation*}
		d_x^T\left(\nabla^2f(x^*)+\sum\limits_{i=1}^m\lambda^*_i\nabla^2 g_i(x^*)+\sum\limits_{i=1}^p\mu^*_i\nabla^2h_i(x^*)\right)d_x\geq 0
	\end{equation*}
	for all $d_x \in \mathcal{C}_\mathcal{X}(x^*)$.
\end{corollary}

\subsection{Second Order Sufficient Optimality Condition}

In this section we state a second order sufficient optimality condition for \eqref{eq:reformulation}. 
We begin by introducing a condition for S-stationary points that can be used to identify which S-stationary points are local minima of \eqref{eq:reformulation}.
Later we also use a similar condition for M-stationary points to give a sufficient condition for the local uniqueness of M-stationary points.

\begin{definition}
	Let $f,g,h$ be twice continuously differentiable, and $(x^*,y^*) \in Z$ be an S-stationary point of \eqref{eq:reformulation}.
	If for all directions $(d_x,d_y)\in \cccc{Z}{x^*,y^*}$ with $d_x\neq0$ there exists \underline{at least one} S-stationary multiplier $(\lambda^*,\mu^*,\gamma^*)$ such that
	\begin{equation}\label{eq:ccossc-inequality}
		d_x^T\left(\nabla^2f(x^*)+\sum\limits_{i=1}^m\lambda^*_i\nabla^2g_i(x^*)+\sum\limits_{i=1}^p\mu^*_i\nabla^2h_i(x^*)\right)d_x>0,
	\end{equation}
	then we say that the \emph{Cardinality Constrained Second Order Sufficient Optimality Condition} (\emph{CC-SOSC}) holds in $(x^*,y^*)$.
\end{definition}

Note that in this condition we do not have to check all $(d_x,d_y) \in \cccc{Z}{x^*,y^*}$ with $(d_x,d_y) \neq 0$ but only those with $d_x \neq 0$.
For directions with $d_x = 0$ condition \eqref{eq:ccossc-inequality} obviously cannot be satisfied.
But since the objective function $f$ does only depend on $x$, the directions $d_x \neq 0$ are the important ones.

From standard nonlinear optimization, we know that a second order sufficiency condition combined with a KKT point yields a strict local minimum.
However, since the objective function here does not depend on $y$, we cannot expect to obtain a strict local minimum with respect to both variables unless $y$ is locally fixed.
For this reason, we have to work with the concept of a strict local minimum with respect to $x$.

\begin{definition}
	We say that a feasible point $(x^*,y^*)$ of \eqref{eq:reformulation} is a \emph{strict local minimum with respect to $x$ of \eqref{eq:reformulation}}, if there exists a radius $r>0$ such that
	\begin{equation*}
		f(x^*)<f(x) \quad\forall (x,y)\in B_r(x^*,y^*)\cap\{(x,y)\in Z \mid x\not=x^*\}.
	\end{equation*}
\end{definition}

Note that a strict local minimum $(x^*,y^*)$ with respect to $x$ is always a local minimum with respect to both variables since for all $(x,y) \in B_r(x^*,y^*)$ either $x = x^*$ and thus $f(x) = f(x^*)$ or $x \neq x^*$ and thus $f(x) > f(x^*)$.

The following theorem shows that CC-SOSC is indeed a sufficient condition for an S-stationary point to be a local minimum of the reformulation \eqref{eq:reformulation}.
For the proof we adapt a line of argument from \cite{HK2007}.

\begin{theorem}[Second Order Sufficient Optimality Condition]\label{thm:cc-ssosc}
	Let $f,g,h$ be twice continuously differentiable and $(x^*,y^*)$ be an S-stationary point of \eqref{eq:reformulation} satisfying CC-SOSC.
	Then $(x^*,y^*)$ is a strict local minimum with respect to $x$ of \eqref{eq:reformulation}.
\end{theorem}

\begin{proof}
	Assume that the claim is false.
	Then we can find a sequence $\left(x^k,y^k\right)_{k}\subseteq Z$ with $(x^k,y^k)\rightarrow (x^*,y^*)$ ($k\rightarrow\infty$) and $x^k\not=x^*$ such that $f(x^k)\leq f(x^*)$ for all $k\in\N$.
	We deduce a contradiction to \eqref{eq:ccossc-inequality} from those properties.
	To this end define the directions $d^k = (d_x^k,d_y^k)$ by
	\begin{align*}
		d_x^k \coloneqq \frac{x^k-x^*}{\vn{x^k-x^*}},\quad\quad d_y^k\coloneqq\frac{y^k-y^*}{\vn{(x^k,y^k)-(x^*,y^*)}}
	\end{align*}
	for all $k\in\N$. 

	We have $\vn{d_x^k}=1$ and $\vn{d_y^k}\leq1$ for all $k\in\N$, i.e. the sequences are bounded. 
	Thus, we can assume without loss of generality that $(d^k)_{k}$ converges to some direction $d = (d_x,d_y)$.
	Because $\vn{d_x^k}=1$ for all $k\in\N$ we know $d_x\not=0$.

	We proceed to show that $d$ is a critical direction.
	To do so, we exploit the fact that $(x^k,y^k)$ are feasible for all $k \in \N$ and converging to $(x^*,y^*)$.
	
	For all $k\in\N$, by the mean value theorem, there is a $\xi^k\in [x^k,x^*]$ such that
	\begin{equation*}
		\nabla g_i(\xi^k)^T(x^k-x^*)=g_i(x^k)-g_i(x^*)\leq 0 \quad \forall i \in I_g(x^*).
	\end{equation*}
	Dividing the above inequality by $\vn{x^k-x^*}$ and letting $k\rightarrow\infty$, we obtain $\nabla g_i(x^*)^Td_x\leq0$ for all $i\in I_g(x^*)$, since $\nabla g_i$ is continuous.
	Analogously we can show $\nabla h_i(x^*)^Td_x=0$ for all $i=1,\dots,p$.

	If $e^T y^*=n-\kappa$, we obtain for all $k \in \N$
	\begin{equation*}
		e^T(y^k-y^*)=e^Ty^k- (n -\kappa) \geq0 
		\quad \Longrightarrow \quad
		e^Td_y\geq0.
	\end{equation*}

	For $i\in I_{\pm0}(x^*,y^*)$ we have $x_i^k\not=0$ and thus $y_i^k=0$ for sufficiently large $k$.
	Hence also
	\begin{equation*}
		e_i^T(y^k-y^*)=y_i^k-y_i^*=0
		\quad \Longrightarrow \quad
		e_i^Td_y=0.
	\end{equation*}

	For $i\in I_{00}(x^*,y^*)$ we have
	\begin{equation*}
		e_i^T(y^k-y^*)=y_i^k-y_i^*=y_i^k\geq0
		\quad \Longrightarrow \quad
		e_i^Td_y \geq0.
	\end{equation*}

	For $i\in I_{01}(x^*,y^*)$ we have
	\begin{equation*}
		e_i^T(y^k-y^*)=y_i^k-y_i^*=y_i^k-1\leq0
		\quad \Longrightarrow \quad
		e_i^Td_y \leq0.
	\end{equation*}

	For $i\in I_{0+}(x^*,y^*)\cup I_{01}(x^*,y^*)$ we have $y_i^k>0$ and thus $x_i^k=0$ for $k$ sufficiently large.
	Hence also
	\begin{equation*}
		e_i^T(x^k-x^*)=x_i^k-x_i^*=0
		\quad \Longrightarrow \quad
		e_i^Td_x \leq0.
	\end{equation*}

	For $i\in I_{00}(x^*,y^*)$ we have 
	\begin{align*}
		(e_i^Td_x)(e_i^Td_y) & =\lim\limits_{k\rightarrow\infty}\left(\frac{e_i^T (x^k-x^*)}{\vn{x^k-x^*}}\right)\left(\frac{e_i^T(y^k-y^*)}{\vn{(x^k,y^k)-(x^*,y^*)}}\right)\\
		& = \lim\limits_{k\rightarrow\infty}\left(\frac{x_i^k-x_i^*}{\vn{x^k-x^*}}\right)\left(\frac{y_i^k-y_i^*}{\vn{(x^k,y^k)-(x^*,y^*)}}\right)\\
		& = \lim\limits_{k\rightarrow\infty}\frac{x_i^k\cdot y_i^k}{\vn{x^k-x^*}\vn{(x^k,y^k)-(x^*,y^*)}}=0.
	\end{align*}
	
	We thus have shown $d\in\cclc{Z}{x^*,y^*}$.
	For all $k\in\N$, applying the mean value theorem to the objective function, we find a $\zeta^k\in [x^k,x^*]$ with
	\begin{equation*}
		\nabla f(\zeta^k)^T(x^k-x^*)=f(x^k)-f(x^*)\leq0
		\quad \Longrightarrow \quad
		\nabla f(x^*)^Td_x\leq0.
	\end{equation*}
	Hence $d\in\cclc{Z}{x^*,y^*}\cap\{(d_x,d_y)\in\R^n\times\R^n\,:\,\nabla f(x^*)^Td\leq0\}=\cccc{Z}{x^*,y^*}$.
	
	Now it remains to show that for all S-stationary multipliers $(\lambda^*,\mu^*,\gamma^*)$ the direction $(d_x,d_y) \in \cccc{Z}{x^*,y^*}$ violates \eqref{eq:ccossc-inequality}.
	To this end fix and arbitrary S-stationary multiplier $(\lambda^*,\mu^*,\gamma^*)$ and define the twice continuously differentiable function $\ell:\R^n\rightarrow\R$ by
	\begin{equation*}
		x\mapsto \ell(x)\coloneqq f(x)+\sum\limits_{i=1}^m\lambda_i^*g_i(x)+\sum\limits_{i=1}^p\mu_i^*h_i(x)+\sum\limits_{i=1}^n\gamma_i^*x_i.
	\end{equation*}
	The Hessian of $\ell$ at $x^*$ is the Hessian in \eqref{eq:ccossc-inequality}. Using the S-stationarity of $(x^*,y^*)$ with the  multipliers $(\lambda^*,\mu^*,\gamma^*)$, we know  $\ell(x^*)=f(x^*)$ and $\nabla \ell(x^*)=0$.

	For sufficiently large $k\in\N$ we thus obtain
	\begin{align}
		\ell(x^*) & =f(x^*)  \geq f(x^k) \notag\\
		& \geq f(x^k) + \sum\limits_{i=1}^m\lambda_i^*g_i(x^k) + \sum\limits_{i=1}^p\mu_i^* h_i(x^k) + \sum\limits_{i=1}^n \gamma_i^*x_i^k = \ell(x^k). \label{thm:ssosc-absch-l}
	\end{align}
	For the second inequality above we use the feasibility of $(x^k,y^k)$ and thus add only non-positive sums.
	The last sum is zero due to the fact that $y_i^k>0$ for all $i\in I_{0+}(x^*,y^*)\cup I_{01}(x^*,y^*)$ and sufficiently large $k\in\N$ and thus $x_i^k=0$. 
	(Note that this argument does not work if $(x^*,y^*)$ is only M-stationary.)
	For each $k\in\N$ Taylor's theorem provides us with a $\xi^k\in [x^k,x^*]$ for which the equality
	\begin{align*}
		\ell(x^k) & = \ell(x^*) + \nabla \ell(x^*)^T(x^k-x^*)+\frac{1}{2}(x^k-x^*)^T\nabla^2\ell(\xi^k)(x^k-x^*)
	\end{align*}
	holds.
	From \eqref{thm:ssosc-absch-l} we know $\ell(x^k)-\ell(x^*)\leq0$.
	Together with $\nabla \ell(x^*)=0$ and the above equality, we therefore have
	\begin{align*} 
		   & (x^k-x^*)^T\left(\nabla^2f(\xi^k)+\sum\limits_{i=1}^n\lambda_i^*\nabla^2g_i(\xi^k)+\sum\limits_{i=1}^p\mu_i^*\nabla^2h_i(\xi^k)\right)(x^k-x^*)\\
		=\ & (x^k-x^*)^T\nabla^2\ell(\xi^k)(x^k-x^*) \ =\  2(\ell(x^k)-\ell(x^*)) \ \leq\ 0
	\end{align*}
	for sufficiently large $k\in\N$. Dividing by $\vn{x^k-x^*}^2$ and letting $k$ tend to infinity this yields a contradiction to the assumption \eqref{eq:ccossc-inequality} due to $d_x \neq 0$. \qed
\end{proof}

In the previous result we have seen that CC-SOSC in an S-stationary point is a sufficient condition for a local minimum.
However, contrary to the corresponding result in nonlinear programming, it guarantees a strict local minimum only with respect to changes in the $x$-variable.
Such a behaviour was to be expected, since the objective function $f$ does not depend on the variable $y$. Thus no point $(x,y)$, at which we can change $y$ without changing $x$, can be a strict local minimum.

This effect can also be observed in the CC-SOSC:
The matrix in \eqref{eq:ccossc-inequality} depends only on the $x$-variable and thus on the $d_x$-part of a critical direction $d = (d_x,d_y)$, whereas the set of critical directions depends on both $x$ and $y$.
For this reason, we have to exclude all critical directions $d =(d_x,d_y)$ with $d_x\not=0$ from the strict inequality \eqref{eq:ccossc-inequality}. 
In contrast, in the SOSC from nonlinear programming and similar results for MPCCs, see for example \cite{NW2006} and \cite{LPR1996}, only the vector $d = (d_x,d_y) = (0,0)$ is excluded from the condition.

Indeed, whenever the cardinality constraint is inactive in a local minimum, one can find critical directions with $d_x=0$, $d_y\not=0$.
For these directions the strict inequality \eqref{eq:ccossc-inequality} cannot hold.
Thus excluding only the vector $(d_x,d_y)=(0,0)$ from \eqref{eq:ccossc-inequality} would lead to a condition which is rarely satisfied.
The following example illustrates this point.

\begin{example}
	Consider the cardinality constrained optimization problem
	\begin{equation*}
		\begin{aligned}
			\min\limits_{x\in\R^2} f(x)=x_1^2+x_2^2 \st g(x)=x_1^2+x_2^2-1\leq0,\ \vn{x}_0\leq1.\\
		\end{aligned}
	\end{equation*}
	The point $x^*=(0,0)$ is the strict (global) minimum of this problem and together with $y^*=(1,0)$ the point $(x^*,y^*)$ is a solution for the continuous reformulation
	\begin{equation*}
		\begin{aligned}
			\min\limits_{(x,y)} f(x)=x_1^2+x_2^2 \st & g(x)=x_1^2+x_2^2-1\leq0,&\\
			& 0\leq y_i\leq1,\ x_i\cdot y_i=0& \forall i=1,2,\\
			& y_1+y_2\geq 1. &
		\end{aligned}
	\end{equation*}
	We have $I_g(x^*)=\emptyset$, $I_0(x^*)=\{1,2\}$ and thus CC-LICQ is fulfilled in $(x^*,y^*)$.
	Hence $(x^*,y^*)$ is an S-stationary point.
	The, due to CC-LICQ unique, S-stationary multipliers are $\lambda^*=\gamma^*=0$.
	Since $\nabla f(x^*)=0$, the critical cone is given by 
	\begin{align*}
		\cccc{Z}{x^*,y^*} = \cclc{Z}{x^*,y^*} = \{(d_x,d_y)\in\R^n\mid & (d_y)_1+(d_y)_2\geq0,\\
		& (d_x)_1=0,\\
		& (d_y)_1\leq0, \quad   (d_y)_2\geq0,\\
		& (d_x)_2\cdot(d_y)_2=0\}
	\end{align*}
	The Hessian in the CC-SOSC condition \eqref{eq:ccossc-inequality} consists only of
	\begin{equation*}
		\nabla^2 f(x^*)=
		\begin{pmatrix}
			2 & 0 \\
			0 & 2
		\end{pmatrix}.
	\end{equation*}
	However, we can choose $(d_x,d_y)=((0,0),(0,1))\in\cccc{Z}{x^*,y^*}\setminus\{0\}$ such that the condition $d_x^T\nabla^2 f(x^*)d_x>0$ is violated.
\end{example}

\subsection{Local uniqueness of M-stationary points using second order information}

While the proofs of Theorems \ref{thm:cc-ssonc} and \ref{thm:cc-ssosc} cannot be transferred directly to M-stationary points of \eqref{eq:reformulation}, we are able to show that an M-stationary point is locally unique, if CC-CPLD and a second order condition hold.
We follow a line of argument by Guo, Lin and Ye \cite{GLY2013}. 
To simplify the presentation of the proof of Theorem \ref{thm:cc-m-isolated} we show the following auxiliary result first.

\begin{proposition}\label{seq-m-multipliers}
	Let $(x^*,y^*)\in Z$ be feasible point of \eqref{eq:reformulation} and $(x^k,y^k)_{k}\subseteq Z$ be a sequence of M-stationary points of \eqref{eq:reformulation} converging to $(x^*,y^*)$.
	\begin{enumerate}[label=(\alph*)]
		\item\label{thm:seq-m-multipliers:part-cc-cpld} If CC-CPLD holds in $(x^*,y^*)$, then $(x^*,y^*)$ is M-stationary and one can find a bounded sequence $(\lambda^k,\mu^k,\gamma^k)_k$ of M-stationary multipliers of $(x^k,y^k)$ such that every accumulation point $(\lambda^*,\mu^*,\gamma^*)$ is an M-stationary multiplier of $(x^*,y^*)$.

		\item If even CC-MFCQ holds in $(x^*,y^*)$, the every sequence $(\lambda^k,\mu^k,\gamma^k)_k$ of M-stationary multipliers of $(x^k,y^k)$ is bounded and every accumulation point $(\lambda^*,\mu^*,\gamma^*)$ is an M-stationary multiplier of $(x^*,y^*)$.
	\end{enumerate}
\end{proposition}

\begin{proof}
	We begin by verifying (a).
	Since $(x^k,y^k)$ are M-stationary points of \eqref{eq:reformulation}, there exist multipliers $(\lambda^k,\mu^k,\gamma^k)$ with
	\begin{align}
		&\nabla f(x^k) + \sum\limits_{i\in I_g(x^k)}\lambda_i^k\nabla g_i(x^k)+\sum\limits_{i=1}^p\mu_i^k\nabla h_i(x^k)+\sum\limits_{i\in I_0(x^k)}\gamma_i^ke_i=0,\label{lem:cc-m-isolated:multiplk1}\\
		&\lambda_i^k\geq0,\quad \lambda_i^kg_i(x^k)=0,\quad\forall i=1,\dots,m,\label{lem:cc-m-isolated:multiplk2}\\
		&\gamma_i^k=0,\quad\forall i\in I_{\pm0}(x^k,y^k).\label{lem:cc-m-isolated:multiplk3}
	\end{align}

	Without loss of generality, we may additionally assume that the vectors
	\begin{equation}\label{eq:cc-m-isolated:independence}
		\nabla g_i(x^k) \  (i \in \supp(\lambda^k)), \quad  \nabla h_i(x^k) \  (i \in \supp(\mu^k)), \quad   e_i \  (i \in \supp(\gamma^k))
	\end{equation}
	are linearly independent.
	Otherwise, the multipliers can be modified according to \cite[Lemma A.1]{SU2010}.

	We show that the sequence $(\lambda^k, \mu^k, \gamma^k)_k$ is bounded and thus has a convergent subsequence.
	To do so, assume for contradiction $\|(\lambda^k, \mu^k, \gamma^k)\| \to \infty$.
	Then the normed sequence
	\[
		\left(\frac{(\lambda^k, \mu^k, \gamma^k)}{\|(\lambda^k, \mu^k, \gamma^k)\|}\right)_{k\in\N}
	\]
	is bounded and thus (at least on a subsequence) convergent to some nonzero limit $(\bar \lambda, \bar \mu, \bar \gamma) \neq 0$.
	This limit then satisfies $\bar \lambda \geq 0$ and $\bar \lambda_i = 0$ for all $i \notin I_g(x^*)$ since then $g_i(x^k) < 0$ and thus $\lambda^k_i = 0$ for all $k$ sufficiently large.
	Similarly, we know $\bar \gamma_i = 0$ for all $i\notin I_0(x^*)$ since then $x^k_i \neq 0$ and thus $\gamma^k_i = 0$ for all $k$ sufficiently large.
	Additionally, we obtain
	\[
	 	\sum\limits_{i\in I_g(x^*)}\bar \lambda_i\nabla g_i(x^*)+\sum\limits_{i=1}^p\bar\mu_i\nabla h_i(x^*)+\sum\limits_{i\in I_0(x^k)}\bar\gamma_i e_i=0
	\] 
	from \eqref{lem:cc-m-isolated:multiplk1}.
	Consequently, the vectors
	\[
		\nabla g_i(x) \  (i \in \supp(\bar \lambda)) \quad \text{and} \quad  \nabla h_i(x) \  (i \in \supp(\bar \mu)), \quad   e_i \  (i \in \supp(\bar \gamma))
	\]
	are positively linearly dependent in $x^*$ and thus by CC-CPLD have to remain linearly dependent in a neighbourhood.
	Due to
	\[
		\supp(\bar \lambda) \subseteq  \supp(\lambda^k), \quad \supp(\bar \mu)) \subseteq  \supp(\mu^k), \quad \supp(\bar \gamma)) \subseteq  \supp(\gamma^k)
	\]
	for all $k$ sufficiently large, we obtain a contradiction to the choice of the multipliers $(\lambda^k, \mu^k, \gamma^k)$ in \eqref{eq:cc-m-isolated:independence}.

	Thus, the sequence $(\lambda^k, \mu^k, \gamma^k)_k$ is bounded and therefore convergent to some limit $(\lambda^*,\mu^*,\gamma^*)$ on a subsequence.

 	Since $f$, $g$ and $h$ are continuously differentiable, we have
 	\[
 		\nabla f(x^*)+\sum\limits_{i=1}^m\lambda_i^*g_i(x^*)+\sum\limits_{i=1}^p\mu_i^*\nabla h_i(x^*)+\sum\limits_{i=1}^n\gamma_i^*e_i = 0.
 	\]
 	Analogously to our previous arguments one sees that $\lambda^* \geq 0$ and $\supp(\lambda^*) \subseteq I_g(x^*)$ as well as $\supp(\gamma^*) \subseteq I_0(x^*)$.
 	Thus, $(x^*,y^*)$ together with the multipliers $(\lambda^*,\mu^*,\gamma^*)$ is M-stationary.

 	To verify part (b) one only has to observe that under the assumption of CC-MFCQ it is not necessary to modify the multipliers to guarantee \eqref{eq:cc-m-isolated:independence} in order to obtain a contradiction. \qed
\end{proof}

The previous result states that the limit of every convergent sequence of M-stationary points is also M-stationary.
This plays a major role in the proof of the following uniqueness theorem for M-stationary points.
In this result, we need an assumption which is closely related to CC-SOSC, but stronger since condition \eqref{eq:ccossc-inequality} now has to hold for all M-stationary multipliers, not only one S-stationary multiplier.
Under CC-LICQ however the M-stationary multiplier is unique.
The following result and its proof is motivated by a similar result for MPCCs \cite{GLY2013}.

\begin{theorem}[Local uniqueness of M-stationary points]\label{thm:cc-m-isolated}
	Let $f,g,h$ be twice continuously differentiable, $(x^*,y^*)$ be an M-stationary point of \eqref{eq:reformulation} satisfying CC-CPLD, and let
	\begin{equation*}
		d_x^T\left(\nabla^2f(x^*)+\sum\limits_{i=1}^m\lambda_i\nabla^2g_i(x^*)+\sum\limits_{i=1}^p\mu_i\nabla^2h_i(x^*)\right)d_x>0
	\end{equation*}
	hold for all $(d_x,d_y)\in\cccc{Z}{x^*,y^*}$ with $d_x\not=0$ and \underline{all} M-stationary multipliers $(\lambda,\mu,\gamma)$ of $(x^*,y^*)$.
	Then there exists a radius $r>0$ such that 
	\begin{equation*}
		\forall (x,y)\in Z \cap B_r(x^*,y^*):\ \left[(x,y) \text{ is M-stationary } \Rightarrow x=x^*\right].
	\end{equation*} 
\end{theorem}

\begin{proof}
	Assume that the claim is false.
	Then there exists a sequence $(x^k,y^k)_{k\in\N}\subset Z$ of M-stationary points converging to $(x^*,y^*)$ with $x^k\not=x^*$.
	By Proposition \ref{seq-m-multipliers}\ref{thm:seq-m-multipliers:part-cc-cpld} we can assume without loss of generality that the corresponding M-stationary multipliers $(\lambda^k,\mu^k,\gamma^k)$ are convergent, too, and that the limit $(\lambda^*,\mu^*,\gamma^*)$ is an M-stationary multiplier for $(x^*,y^*)$, i.e.
	\begin{align}
		&\nabla f(x^*)+\sum\limits_{i=1}^m\lambda_i^*\nabla g_i(x^*)+\sum\limits_{i=1}^p\mu_i^*\nabla h_i(x^*) + \sum\limits_{i=1}^n \gamma_i^*e_i = 0, \notag\\
		&g_i(x^*)\leq0,\ \lambda_i^*\geq0,\ \lambda_i^*g_i(x^*)=0,\quad\forall i=1,\dots,m,\label{thm:cc-m-isolated:multipl*} \\
		&h_i(x^*)=0,\quad\forall i=1,\dots,p,\notag\\
		&\gamma_i^*=0,\quad\forall i\in I_{\pm0}(x^*,y^*).\notag
	\end{align}
	For $k\in\N$ define the directions $d^k = (d_x^k,d_y^k)$ by
	\begin{align*}
		d_x^k \coloneqq \frac{x^k-x^*}{\vn{x^k-x^*}},\quad\quad d_y^k\coloneqq\frac{y^k-y^*}{\vn{(x^k,y^k)-(x^*,y^*)}}.
	\end{align*}
	We have $\vn{d_x^k}=1$ and $\vn{d_y^k}\leq1$ for all $k\in\N$.
	Hence the sequences are bounded and we can assume without loss of generality that $d^k = (d_x^k,d_y^k)$ is convergent.
	Denote the limit by $d = (d_x,d_y)$.
	Since $\vn{d_x^k}=1$ for all $k\in\N$, we have $d_x\not=0$.
	
	Furthermore we can show $(d_x,d_y)\in\cclc{Z}{x^*,y^*}$ analogously to the proof of Theorem \ref{thm:cc-ssosc}.
	
	Before we show $\nabla f(x^*)^Td_x\leq0$, we derive four helpful equations.
	Since $(\lambda^k,\mu^k,\gamma^k)$ is a M-stationary multiplier for $(x^k,y^k)$, we have
	\begin{equation}\label{thm:cc-m-isolated:lagrangekk}
		\sum\limits_{i=1}^m\lambda_i^k g_i(x^k)+\sum\limits_{i=1}^p\mu_i^k h_i(x^k)+\sum\limits_{i=1}^n\gamma_i^k x_i^k=0
	\end{equation}
	for all $k\in\N$. 
	Because of the continuity of $g_i$ and the properties of the multipliers $(\lambda^k,\mu^k,\gamma^k)$ and $(\lambda^*,\mu^*,\gamma^*)$, the implications 
	\begin{align*} 
		g_i(x^*)\not=0\ \Rightarrow\ g_i(x^k)\not=0\ \Rightarrow\ \lambda_i^k=0, \\
		x_i^*\not=0\ \Rightarrow\ x_i^k\not=0\ \Rightarrow\ \gamma_i^k=0,\\
		\lambda_i^*\not=0\ \Rightarrow\ \lambda_i^k\not=0\ \Rightarrow\ g_i(x^k)=0,\\
		\gamma_i^*\not=0,\ \Rightarrow\ \gamma_i^k\not=0\ \Rightarrow\ x_i^k=0,
	\end{align*}
	hold for sufficiently large $k$.
	Hence we also have
	\begin{align}
		\label{thm:cc-m-isolated:lagrangek*} \sum\limits_{i=1}^m\lambda_i^k g_i(x^*)+\sum\limits_{i=1}^p\mu_i^k h_i(x^*)+\sum\limits_{i=1}^n\gamma_i^k x_i^*=0,\\
		\label{thm:cc-m-isolated:lagrange*k} \sum\limits_{i=1}^m\lambda_i^* g_i(x^k)+\sum\limits_{i=1}^p\mu_i^* h_i(x^k)+\sum\limits_{i=1}^n\gamma_i^* x_i^k=0,
	\end{align}
	for all $k$ sufficiently large.
	Define $\ell:\R^n\times\R^m\times\R^p\times\R^n\rightarrow\R$ by 
	\begin{equation*}
		(x,\lambda,\mu,\gamma)\mapsto \ell(x,\lambda,\mu,\gamma)\coloneqq \sum\limits_{i=1}^m\lambda_i g_i(x)+\sum\limits_{i=1}^p\mu_i h_i(x)+\sum\limits_{i=1}^n\gamma_i x_i.
	\end{equation*}
	
	A first order Taylor-expansion of $x\mapsto\ell(x,\lambda^k,\mu^k,\gamma^k)$ around $x^*$ evaluated at $x^k$ yields
	\begin{align*}
		0 & = \ell(x^k,\lambda^k,\mu^k,\gamma^k)\\
		& = \ell(x^*,\lambda^k,\mu^k,\gamma^k)+\nabla \ell(x^*,\lambda^k,\mu^k,\gamma^k)(x^k-x^*)+\text{o}(\vn{x^k-x^*})\\
		& = \nabla \ell(x^*,\lambda^k,\mu^k,\gamma^k)(x^k-x^*)+\hbox{o}(\vn{x^k-x^*})
	\end{align*}
	for sufficiently large $k$.
	here, we used \eqref{thm:cc-m-isolated:lagrangekk} and \eqref{thm:cc-m-isolated:lagrangek*}.
	By dividing through $\vn{x^k-x^*}$ and letting $k$ tend to infinity we get
	\begin{equation}\label{thm:cc-m-isolated:gradient-l-**}
		0 = \nabla \ell(x^*,\lambda^*,\mu^*,\gamma^*)d_x.
	\end{equation}
	Using this together with the M-stationarity of $(x^*,y^*)$ we can calculate
	\begin{align*}
		\nabla f(x^*)d_x & = -\left(\sum\limits_{i=1}^m\lambda_i^*\nabla g_i(x^*)+\sum\limits_{i=1}^p\mu_i^*\nabla h_i(x^*) + \sum\limits_{i=1}^n \gamma_i^*e_i^T\right)d_x\\
		& = -\nabla\ell(x^*,\lambda^*,\mu^*,\gamma^*))d_x = 0.
	\end{align*}
	Because $(d_x,d_y)\in\cclc{Z}{x^*,y^*}$ we consequently have verified $(d_x,d_y)\in \cccc{Z}{x^*,y^*}$.
	
	To keep the notation more compact, define $\omega$ as an abbreviation for the multipliers $\omega := (\lambda, \mu, \gamma)$.
	For $k\in\N$ define the functions
	\begin{align*}
		\bar x^k:[0,1]\rightarrow \R^n,&\quad t\mapsto\bar x(t)\coloneqq x^*+t\cdot(x^k-x^*),\\
		\bar \omega^k:[0,1]\rightarrow \R^{m+p+n},&\quad t\mapsto\bar \omega(t)\coloneqq \omega^*+t\cdot(\omega^k-\omega^*),
	\end{align*}
	and $s_k:[0,1]\rightarrow \R$ by
	\begin{align*}
		s_k(t)\coloneqq \big(\nabla f(\bar x^k(t))+\nabla\ell\left(\bar x^k(t),\bar \omega^k(t)\right)\big)^T(x^k-x^*)-\ell\left(\bar x^k(t),\omega^k-\omega^*\right).
	\end{align*}
	Using \eqref{thm:cc-m-isolated:multipl*}-\eqref{thm:cc-m-isolated:lagrange*k} and the fact that $\omega^k = (\lambda^k,\mu^k,\gamma^k)$ is an M-stationary multiplier for $(x^k,y^k)$ we can calculate
	\begin{align*}
		s_k(0)  =&  \left(\nabla f(x^*)+\nabla\ell(x^*,\omega^*)\right)^T(x^k-x^*)-\ell(x^*,\omega^k-\omega^*)\\
		= & -\left(\sum\limits_{i=1}^m\lambda_i^k g_i(x^*)+\sum\limits_{i=1}^p\mu_i^k h_i(x^*)+\sum\limits_{i=1}^n\gamma_i^k x_i^*\right)\\
		&-\left(\sum\limits_{i=1}^m\lambda_i^* g_i(x^*)+\sum\limits_{i=1}^p\mu_i^*h_i(x^*)+\sum\limits_{i=1}^p\gamma_i^*x_i^*\right)\  = 0,\\
		s_k(1) =& \left(\nabla f(x^k)+\nabla\ell(x^k,\omega^k)\right)^T(x^k-x^*)-\ell(x^k,\omega^k-\omega^*)\\
		= & -\left(\sum\limits_{i=1}^m\lambda_i^k g_i(x^k)+\sum\limits_{i=1}^p\mu_i^kh_i(x^k)+\sum\limits_{i=1}^n\gamma_i^k x_i^k\right)\\
		& -\left(\sum\limits_{i=1}^m\lambda_i^*g_i(x^k)+\sum\limits_{i=1}^p\mu_i^*h_i(x^k)+\sum\limits_{i=1}^n\gamma_i^*x_i^k\right) \  = 0.
	\end{align*}
	The functions $s_k$ are twice continuously differentiable.
	The mean value theorem provides the existence of a $\tau_k\in(0,1)$ such that
	\begin{equation}\label{thm:cc-m-isolated:sk'}
		s_k'(\tau_k)=\frac{s_k(1)-s_k(0)}{1-0}=0.
	\end{equation}
	Using $(\bar x^k)'(t)=x^k-x^*$ and $(\bar \omega^k)'(t)=\omega^k-\omega^*$ it is straight forward to calculate
	{\small\begin{equation*}
		s_k'(\tau_k)=(x^k-x^*)^T\left(\nabla^2f(\bar x^k(\tau_k))+\sum\limits_{i=1}^m\bar\lambda_i^k(\tau_k)\nabla^2g_i(\bar x^k(\tau_k))+\sum\limits_{i=1}^p\mu_i^k(\tau_k)\nabla^2h_i(\bar x^k(\tau_k))\right)(x^k-x^*).
	\end{equation*}}
	Since $\tau_k$ is bounded we have $\bar x^k(\tau_k)\rightarrow x^*$, $\bar \omega^k(\tau_k) \rightarrow \omega^*$ for $k\rightarrow\infty$.
	It follows from \eqref{thm:cc-m-isolated:sk'} that
	{\small\begin{equation*}
		\frac{(x^k-x^*)^T}{\vn{x^k-x^*}}\left(\nabla^2f(\bar x^k(\tau))+\sum\limits_{i=1}^m\bar\lambda_i^k(\tau)\nabla^2g_i(\bar x^k(\tau))+\sum\limits_{i=1}^p\mu_i^*\nabla^2h_i(\bar x^k(\tau))\right)\frac{(x^k-x^*)}{\vn{x^k-x^*}} = 0
	\end{equation*}}
	for sufficiently large $k\in\N$ and thus for $k\rightarrow\infty$
	\begin{equation*}
		d_x^T\left(\nabla^2f(x^*)+\sum\limits_{i=1}^m\lambda_i^*\nabla^2g_i(x^*)+\sum\limits_{i=1}^p\mu_i^*\nabla^2h_i(x^*))\right)d_x = 0,
	\end{equation*}
	since the functions $f$, $g$ and $h$ are twice continuously differentiable.
	Because $(d_x,d_y)\in\cccc{Z}{x^*,y^*}$ and $d_x\not=0$, this is a contradiction to the theorem's assumption. \qed
\end{proof}

Since the definition of an M-stationary point is independent from $y$, we can formulate a result on uniqueness of M-stationary points directly for \eqref{eq:ccproblem}.

\begin{corollary}\label{cor:cc-m-isolated}
	Let $f,g,h$ be twice continuously differentiable. Let $x^*$ be feasible for \eqref{eq:ccproblem}, M-stationary, satisfy CC-CPLD and let
	\begin{equation*}
		d_x^T\left(\nabla^2f(x^*)+\sum\limits_{i=1}^m\lambda_i\nabla^2g_i(x^*)+\sum\limits_{i=1}^p\mu_i\nabla^2h_i(x^*)\right)d_x>0
	\end{equation*}
	hold for all $d_x\in \mathcal{C}_\mathcal{X}(x^*)$ with $d_x\not=0$ and all M-stationary multipliers $(\lambda,\mu,\gamma)$ of $x^*$.
	Then there exists a radius $r>0$ such that 
	\begin{equation*}
		\forall (x,y)\in Z \cap (B_r(x^*) \times \R^n):\ \left[(x,y) \text{ is M-stationary } \Rightarrow x=x^*\right].
	\end{equation*} 
\end{corollary}

\begin{proof}
	For every $\bar y$, such that $(x^*,\bar y) \in Z$, the point $(x^*,\bar y)$ is M-stationary for \eqref{eq:reformulation}.
	Due to the definition of $d_x\in \mathcal{C}_\mathcal{X}(x^*)$ the prerequisites of Theorem \ref{thm:cc-m-isolated} are satisfied and thus there exists $r_{\bar y} > 0$ such that
	\begin{equation*}
		\forall (x,y)\in Z \cap B_{r_{\bar y}}(x^*,\bar y) :\ \left[(x,y) \text{ is M-stationary } \Rightarrow x=x^*\right].
	\end{equation*} 

	Together the balls $B_{r_{y}}(x^*,y)$ form an open covering of the compact set $\{(x^*,y) \mid (x^*,y) \in Z\}$ and thus we can find a $r > 0$ such that \begin{equation*}
		\forall (x,y)\in Z \cap B_{r}(x^*,\bar y) :\ \left[(x,y) \text{ is M-stationary } \Rightarrow x=x^*\right]
	\end{equation*} 
	for all $(x^*,\bar y) \in Z$.

	Now consider an arbitrary M-stationary point $(x,y)\in Z \cap (B_r(x^*) \times \R^n)$.
	By choosing $r > 0$ sufficiently small we can ensure the implication
	\[
		x_i^* \neq 0 \quad \Longrightarrow \quad x_i \neq 0 \quad \Longrightarrow \quad y_i = 0
	\]
	and thus $(x^*, y) \in Z$.
	This implies $(x,y) \in B_r(x^*,y)$ and thus $x = x^*$. \qed
\end{proof}

This result is later used to ensure the local convergence of a Scholtes-type regularization method.

\section{Convergence Properties of Scholtes Regularization}\label{sec:scholtes}

The Scholtes regularization for MPCCs \cite{S2001} has been successfully adapted to the relaxation of cardinality constrained and sparse optimization problems \eqref{eq:reformulation} in \cite{BBCS2017,FMPSW2013}.
As in the MPCC case, the adapted version is numerically very successful compared to other regularization approaches.
In this section we briefly introduce the Scholtes regularization for cardinality constrained optimization problems investigated already in \cite{BBCS2017}.
We also repeat convergence results for this regularization. 
Then we use the second order optimality conditions from Section \ref{sec:2nd-order} to expand the convergence theory. 
We show that the regularized programs have a solution in a neighbourhood of a strict local minimum $x^*$ of \eqref{eq:ccproblem}.
We then use that result to prove the convergence of KKT points $(x^k,y^k$) of the regularized programs to $x^*$.

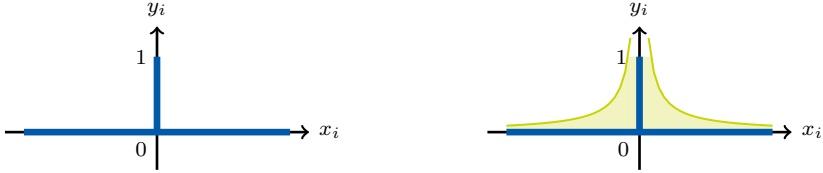
\begin{figure}[h]
	\centering
	\begin{minipage}{.4\textwidth}
		\begin{tikzpicture}[scale=1]
			\tikzstyle{every node}=[font=\footnotesize]
			\draw[->,line width = 1pt] (-2,0) -- (0,0) node[below left]{$0$} -- (2,0) node[right]{$x_i$};
			\draw[->,line width = 1pt] (0,-0.5) -- (0,1) node[left]{$1$} -- (0,1.4) node[above]{$y_i$};
			\draw[-, tud1b, line width = 2.5pt] (-1.75,0) -- (1.75,0);
			\draw[-, tud1b, line width = 2.5pt] (0,0) -- (0,1);
		\end{tikzpicture}
	\end{minipage}
	\quad
	\begin{minipage}{.4\textwidth}
		\begin{tikzpicture}[scale=1]
			\tikzstyle{every node}=[font=\footnotesize]
			\draw[->,line width = 1pt] (-2,0) -- (0,0) node[below left]{$0$} -- (2,0) node[right]{$x_i$};
			\draw[->,line width = 1pt] (0,-0.5) -- (0,1.4) node[above]{$y_i$};
			\fill[fill=tud5b, opacity = 0.25] (-1.75,0) -- plot [domain=-1.75:-0.15] (\x,{-0.15*\x^(-1)}) -- (-0.15,0) -- plot[domain=-0.15:0.15] (\x,1) -- (0.15,0) -- plot [domain=0.15:1.75] (\x,{0.15*\x^(-1)}) -- (1.75,0) -- cycle;
			\draw[-, tud5b, line width = 0.75pt] plot [domain=-1.75:-0.12] (\x,{-0.15*\x^(-1)});
			\draw[-, tud5b, line width = 0.75pt] plot [domain=0.12:1.75] (\x,{0.15*\x^(-1)});
			\draw[-, tud1b, line width = 2.5pt] (-1.75,0) -- (1.75,0);
			\draw[-, tud1b, line width = 2.5pt] (0,0) -- (0,1) node[left,black]{$1$};			
		\end{tikzpicture}
	\end{minipage}
	\caption{Orthogonality constraints (left) and the Scholtes-type regularization (right)}\label{fig:regularization-scholtes}
\end{figure}

To adapt the Scholtes regularization originally introduced for MPCCs in \cite{S2001} to \eqref{eq:reformulation}, the orthogonality constraint $x_i\cdot y_i=0$, $i=1,\dots,n$, is replaced by 
\begin{equation*}
	-t\leq x_i\cdot y_i\leq t\quad\forall i=1,\dots,n,
\end{equation*}
for a regularization parameter $t\geq 0$, see Figure \ref{fig:regularization-scholtes} for an illustration.
The resulting regularized programs are given by
\begin{equation}\label{eq:reg-nlp-scholtes}
	\begin{aligned}
		\text{NLP($t$)}:\quad\quad\min\nolimits_{(x,y)\in\R^n\times\R^n}\ f(x) \st & g(x)\leq0, \quad h(x) = 0, &\\
		& 0 \leq y \leq e, \quad e^Ty \geq n - \kappa, &\\
		& -t e \leq x \circ y\leq te.&
	\end{aligned}
\end{equation}
Let $Z(t)$ be the feasible set of NLP($t$) for $t\geq0$.
The idea of the regularization method is to compute KKT points of NLP($t$) for decreasing parameters $t\rightarrow0$ to obtain a feasible and stationary point of \eqref{eq:reformulation}.%
In \cite{BBCS2017} it was shown that the limit of such a sequence is S-stationary under CC-MFCQ.
We repeat the precise result here for completeness sake.

\begin{theorem}[{\cite[Theorem 3.1]{BBCS2017}}]\label{thm:reg-scholtes-convergence}
   Let $(t^k)_{k} \downarrow 0$ and $(x^k,y^k)_{k}$ be a sequence of KKT points of \textnormal{NLP(}$t^k$\textnormal{)} with $x^k \to x^*$.
   If CC-MFCQ holds at $x^*$, then for every accumulation point $y^*$ of  the bounded sequence $(y^k)_k$  the pair $(x^*,y^*)$ is an S-stationary point of \eqref{eq:reformulation}.
\end{theorem}

A necessary follow up question is whether the regularized programs NLP($t$) possess KKT points.
In \cite{BBCS2017} it was shown that the regularized problems satisfy standard MFCQ if the original problem \eqref{eq:reformulation} satisfies CC-MFCQ.
We state a slightly modified version here, whose proof coincides exactly with the one of \cite[Theorem 3.2]{BBCS2017}.

\begin{theorem}[{\cite[Theorem 3.2]{BBCS2017}}]\label{thm:reg-scholtes-mfcq}
   Let $(x^*,y^*)$ be feasible for \eqref{eq:reformulation} and CC-MFCQ hold there.
   Then there exists a radius $r > 0$ and  a $T > 0$ such that for all $t \in (0,T]$ standard MFCQ for NLP($t$) holds at every $(x,y)\in Z(t)$ with $x \in B_r(x^*)$.
\end{theorem}
Thus in case the regularized problems have local solutions close to $(x^*,y^*)$, these local solutions are KKT points.
Next, we investigate whether the regularized programs NLP($t$) posses a local solution in the vicinity of a local solution $x^*$ of \eqref{eq:ccproblem}.

\begin{theorem}\label{thm:reg-scholtes-solutions}
	\begin{enumerate}[label=(\alph*)]
		\item\label{thm:reg-scholtes-solutions:part-min-ccproblem} Let $x^*$ be a strict local minimum  of \eqref{eq:ccproblem}.
		Then there exist $r>0$ and $T>0$ such that for all $t \in (0,T]$ the regularized program NLP($t$) has a local minimum $(x^t,y^t)$ with $x^t \in B_r(x^*)$.

		\item\label{thm:reg-scholtes-solutions:part-min-reformulation} Let $(x^*,y^*)$ be a strict local minimum of \eqref{eq:reformulation} with respect to $x$ and $\|x^*\|_0 = \kappa$.
		Then there exist $r>0$ and $T>0$ such that for all $t \in (0,T]$ the regularized program NLP($t$) has a local minimum $(x^t,y^t)$ with $x^t \in B_r(x^*)$.
	\end{enumerate}
\end{theorem}

\begin{proof}
	(a)\quad  By assumption there exists a radius $r> 0$ such that for all $x \in \overline{B_r(x^*)} \setminus \{x^*\}$ feasible for \eqref{eq:ccproblem} we have $f(x) > f(x^*)$.

	Now assume that there is no $T > 0$ such that NLP($t$) has a local minimum in $Z(t) \cap (B_r(x^*) \times \R^n)$ for all $t \in (0,T]$.
	Then we can find a sequence $t_k \downarrow 0$ such that NLP($t_k$) has no local minimum on $Z(t_k) \cap (B_r(x^*) \times \R^n)$.
	Since the set $Z(t_k) \cap (\overline{B_r(x^*)} \times \R^n)$ is nonempty and compact (recall that the $y$-variables are always bounded), $f$ attains a global minimum $(x^k,y^k)$ there.
	Consequently $x^k \in \partial B_r(x^*)$ and $f(x^k) < f(x^*)$.
	If we had $f(x^k)\geq f(x^*)$, then the point $(x^*,y^*)$, where $y^*_i=0$, for all $i\in\supp(x^*)$ and $y_i^*=1$ for all $i\in I_0(x^*)$, would be a local minimum of $f$ on $Z(t_k) \cap (B_r(x^*) \times \R^n)$.

	Since $\partial B_r(x^*)$ is compact, we may assume without loss of generality that $(x^k)_k$ converges to some limit $\bar x\in\partial B_r(x^*)$, which implies $\bar x \neq x^*$.
	And since $y^k$ is bounded, it is also convergent (at least on a subsequence) to some limit $\bar y$.
	Letting $t_k \downarrow 0$, we obtain $(\bar x,\bar y) \in Z$. 
	Hence $\bar x$ is feasible for \eqref{eq:ccproblem}.
	Due to $\bar x \neq x^*$ and the choice of $r$, this yields the contradiction
	\[
		f(x^*) \geq \lim_{k \to \infty} f(x^k) = f(\bar x) > f(x^*).
	\]

	\smallskip
	(b)\quad We only have to show that the assumptions imply that $x^*$ is a strict local minimum of \eqref{eq:ccproblem}.
	To this end consider an arbitrary sequence $x^k \to x^*$ feasible for \eqref{eq:ccproblem} with $x^k \neq x^*$.
	Because $x^k$ is feasible for \eqref{eq:ccproblem}, the active cardinality constraint $\|x^*\|_0 = \kappa$ implies that $(x^k,y^*)$ is feasible for \eqref{eq:reformulation} for all $k$ sufficiently large.
	Consequently we have $f(x^k) > f(x^*)$, due to $x^k \neq x^*$.
	By part \ref{thm:reg-scholtes-solutions:part-min-ccproblem} there exist $r>0$, $T>0$ such that for all $t\in(0,T]$ NLP($t$) has a local minimum $(x^t,y^t)$ with $x^t \in B_r(x^*)$. \qed
\end{proof}

If $(x^*,y^*)$ is a strict local minimum of the reformulation \eqref{eq:reformulation} with respect to $x$ but the cardinality constraint is not active, then Theorem \ref{thm:reg-scholtes-solutions} does not guarantee the existence of solutions of NLP($t$) in a neighbourhood unless $x^*$ is a strict local minimum of the original problem \eqref{eq:ccproblem}.
This is in fact an advantage because local minima of the reformulation \eqref{eq:reformulation}, in which the cardinality constraint is not active, are not necessarily local minima of the original problem \eqref{eq:ccproblem} and thus not points we want the regularization method to converge to.
Precisely this situation is illustrated in the following example.

\begin{example}\label{ex:card-strict-necessary}
	Consider the cardinality constrained optimization problem
	\begin{equation*}\label{ex:card-strict-necessary:ccop}
			\min\limits_{x\in\R^3} f(x)=\vn{x-(0,1,2)^T}^2 \st \vn{x}_0\leq1.
	\end{equation*}
	Then $x^1=(0,0,2)^T$ is the global minimum, $x^2=(0,1,0)^T$ is a local minimum, but $x^*=(0,0,0)^T$ is no local minimum.
	Now consider the continuous reformulation, which is is given by
	\begin{equation*}\label{ex:card-strict-necessary:relaxation}
		\min\limits_{x\in\R^2} f(x)=\vn{x-(0,1,2)^T}^2 \st 0 \leq y \leq e, \quad  e^Ty \geq 2,  \quad x \circ y = 0.
	\end{equation*}
	Then choosing $y^*=(1,1,1)^T$ the point $(x^*,y^*)$ is a strict local minimum of the continuous reformulation with respect to $x$ since for all $r \in (0,1)$ all points $(x,y) \in B_r(x^*,y^*) \cap Z$ satisfy $x = x^*$.
	The regularized program for a parameter $t>0$ is given by
	\begin{equation*}\label{ex:card-strict-necessary:scholtes-regularization}
			\min\limits_{x\in\R^2} f(x)=\vn{x-(0,1,2)^T}^2 \st 0 \leq y \leq e, \quad  e^Ty \geq 2,  \quad -te \leq x \circ y \leq te.
	\end{equation*}
	For all $(x,y)\in Z(t)$ sufficiently close to $(x^*,y^*)$ we have $y_i>0$ and $e^Ty>2$.
	Hence in a neighbourhood of $(x^*,y^*)$ the KKT-conditions of the regularized program in $(x,y)$ imply
	\begin{align*}
		0 & = 2(x_2-1)+\gamma_2y_2 \quad \Longrightarrow \quad \gamma_2y_2 \approx 2,\\
		0 & = 2(x_3-2)+\gamma_3y_3\quad \Longrightarrow \quad \gamma_3y_3 \approx 4,\\
		0 & = \nu+\gamma \circ x , \quad \nu \geq 0,  \quad \gamma \circ x \geq 0
	\end{align*}
	Here, the last equation implies $\nu = 0$ and $\gamma \circ x = 0$, which is only possible if $\gamma = 0$.
	This, however, is a contradiction to the first two equations.
	Thus the KKT-conditions cannot be satisfied in a neighbourhood of $(x^*,y^*)$. 
	Since CC-LICQ holds in $(x^*,y^*)$, it follows from Theorem \ref{thm:reg-scholtes-mfcq} that MFCQ holds for the regularized problem  sufficiently close to $(x^*,y^*)$.
	Consequently the regularized program cannot have local minima in a vicinity of $(x^*,y^*)$.
	
	This implies that the Scholtes-type regularization cannot converge to the undesirable local solution $(x^*,y^*)$ of the continuous reformulation , which does not correspond to a solution of the original problem.
\end{example}

Combining all of our previous results, we are now able to prove the main result of this section:
Whenever $x^*$ is a strict local minimum of \eqref{eq:ccproblem} satisfying CC-MFCQ, then the Scholtes relaxation method is locally well defined and the KKT points $(x^k,y^k)$ converge to $x^*$ at least in the $x$-component.
If additionally $\|x^*\|_0 = \kappa$ holds, then the $y$-component is also convergent.

\begin{theorem}\label{thm:reg-scholtes-x-unique}
	\begin{enumerate}[label=(\alph*)]
		\item Let $x^*$ be a strict local minimizer of \eqref{eq:ccproblem} (or $(x^*,y^*)$ be a strict local minimum of \eqref{eq:reformulation} with respect to $x$ and $\|x^*\|_0 = \kappa$) such that CC-MFCQ holds in $x^*$.
		Then there exist $T>0$ and $r > 0$ such that for all $t \in (0,T]$ NLP($t$) has a local minimum/KKT point $(x^t,y^t)$ with $x^t \in B_r(x^*)$.

		\item Let $(x^*,y^*) \in Z$ satisfy CC-MFCQ and choose $r > 0$ sufficiently small.
		Consider a sequence $(t_k)_k \downarrow 0$ and KKT points $(x^k,y^k)_k$ of NLP($t_k$) such that $x^k \in B_r(x^*)$ for all $k \in \N$.
		Then the sequence $(x^k,y^k)_k$ has accumulation points and every accumulation point $(\bar x, \bar y)$ is an S-stationary point of \eqref{eq:reformulation}.

		\item\label{thm:reg-scholtes-x-unique:part-sosc-xy} Let $f,g,h$ be twice continuously differentiable.
		Let $(x^*,y^*)$ be a strict local minimum of \eqref{eq:reformulation} with respect to $x$ and $\|x^*\|_0 = \kappa$ such that CC-MFCQ holds and
		\begin{equation*}
			d_x^T\left(\nabla^2f(x^*)+\sum\limits_{i=1}^m\lambda_i\nabla^2g_i(x^*)+\sum\limits_{i=1}^p\mu_i\nabla^2h_i(x^*)\right)d_x>0
		\end{equation*}
		hold for all $(d_x,d_y)\in\cccc{Z}{x^*,y^*}$ with $d_x\not=0$ and all S-stationary multipliers $(\lambda,\mu,\gamma)$ of $(x^*,y^*)$. 
		Then there exists $r >0$ such that for all sequences $(t_k)_k \downarrow 0$ for all $k$ sufficiently large NLP($t_k$) has a KKT point $(x^k,y^k)$ with $x^k \in B_r(x^*)$ and $(x^k,y^k) \to (x^*,y^*)$.

		\item\label{thm:reg-scholtes-x-unique:part-sosc-x} Let $f,g,h$ be twice continuously differentiable.
		Let $x^*$ be a  strict local minimum of \eqref{eq:ccproblem} such that CC-MFCQ holds and
		\begin{equation*}
			d_x^T\left(\nabla^2f(x^*)+\sum\limits_{i=1}^m\lambda_i\nabla^2g_i(x^*)+\sum\limits_{i=1}^p\mu_i\nabla^2h_i(x^*)\right)d_x>0
		\end{equation*}
		hold for all $d_x\in \mathcal{C}_\mathcal{X}(x^*)$ with $d_x\not=0$ and all M-stationary multipliers $(\lambda,\mu,\gamma)$ of $x^*$. 
		Then there exists $r >0$ such that for all sequences $(t_k)_k \downarrow 0$ for all $k$ sufficiently large NLP($t_k$) has a KKT point $(x^k,y^k)$ with $x^k \in B_r(x^*)$ and $x^k \to x^*$.
	\end{enumerate}
\end{theorem}

\begin{proof}
	(a)\quad By Theorem \ref{thm:reg-scholtes-solutions} the assumptions guarantee the existence of $T > 0$ and $r > 0$ such that for all $t \in (0,T]$ NLP($t$) has a local minimum $(x^t,y^t)$ with $x^t \in B_r(x^*)$.
	Decreasing $T$ and $r$ if necessary we can also use Theorem \ref{thm:reg-scholtes-mfcq}, which guarantees MFCQ for NLP($t$) in $(x^t,y^t)$ and thus ensures that $(x^t,y^t)$ are KKT points.

	\smallskip
	(b)\quad Since $x^k \in B_r(x^*)$ and $y^k \in [0,e]$ for all $k \in \N$ the sequence $(x^k,y^k)$ is bounded and thus has at least one accumulation point.
	Now consider an arbitrary accumulation point $(\bar x, \bar y)$.
	Since $t_k \downarrow 0$ we know that $(\bar x, \bar y)$ is feasible for \eqref{eq:reformulation}.
	If we chose $r > 0$ small enough Remark \ref{rem:cc-mfcq-neighborhood} tells us that CC-MFCQ in $x^*$ implies CC-MFCQ in $(\bar x, \bar y)$.
	Thus by Theorem \ref{thm:reg-scholtes-convergence} $(\bar x, \bar y)$ is an S-stationary point of \eqref{eq:reformulation}.

	\smallskip
	(c)\quad Combining part (a) and (b), we see that there exists $r > 0$ such that for all $t_k > 0$ sufficiently small NLP($t_k$) has a KKT point $(x^k,y^k)$ with $x^k \in B_r(x^*)$ and that all accumulation points $(\bar x, \bar y)$ of $(x^k,y^k)_k$ are S-stationary and thus M-stationary.
	By choosing $r > 0$ small enough, we can enforce $\bar y = y^*$.
	Since $(x^*,y^*)$ is a local minimum satisfying CC-MFCQ, it is an S-stationary point, too.
	Furthermore, since $\|x^*\|_0 = \kappa$, S- and M-stationarity coincide and thus $(x^*,y^*)$ satisfies the assumptions for Theorem \ref{thm:cc-m-isolated}.
	Thus, if we choose $r > 0$ small M-stationarity of the accumulation points $(\bar x, \bar y) = (\bar x, y^*)$ implies $\bar x = x^*$.
	This shows $x^k \to x^*$ and  $y^k \to y^*$.

	\smallskip
	(d)\quad Since $x^*$ is a local minimum of \eqref{eq:ccproblem}, every $(x^*,y) \in Z$ is a local minimum of \eqref{eq:reformulation} and, due to CC-MFCQ, an S-stationary and thus M-stationary point.
	Furthermore, the set of M-stationary multipliers is independent from $y$.
	Using the assumptions, we obtain from Corollary \ref{cor:cc-m-isolated} that there exists an $r > 0$ such that all M-stationary points $(\bar x, \bar y) \in Z$ with $\bar x \in B_r(x^*)$ satisfy $\bar x = x^*$.
	Analogously to (c) we see that we can decrease $r > 0$ such that for all $t_k > 0$ sufficiently small NLP($t_k$) has a KKT point $(x^k,y^k)$ with $x^k \in B_r(x^*)$ and that all accumulation points $(\bar x, \bar y)$ of $(x^k,y^k)_k$ are S-stationary and thus M-stationary.
	Consequently all accumulation points satisfy $\bar x = x^*$ which shows $x^k \to x^*$. \qed
\end{proof}

Note that the second order condition in part \ref{thm:reg-scholtes-x-unique:part-sosc-xy} and \ref{thm:reg-scholtes-x-unique:part-sosc-x} is automatically satisfied if $f$ is uniformly convex, $g$ convex and $h$ affine linear.
Furthermore, in part \ref{thm:reg-scholtes-x-unique:part-sosc-xy} the additional assumption $\vn{x^*}_0=\kappa$ implies by \cite[Theorem 3.6.]{BKS2015} that the vector $x^*$ is a strict local minimum of the cardinality constraint problem \eqref{eq:ccproblem}.
Combining this with a few other previously used arguments, once can alternatively prove part \ref{thm:reg-scholtes-x-unique:part-sosc-xy} by showing that it is implied by part \ref{thm:reg-scholtes-x-unique:part-sosc-x}.

\section{Conclusion}

We discussed a reformulation of cardinality constrained optimization problems using continuous auxiliary variables.
Our article contains three main results on second order conditions for this reformulation:
Second order necessary and sufficient optimality conditions for S-stationary points, as well as a uniqueness result for M-stationary points.
All of these second order conditions capture the lack of curvature of the objective function regarding the auxiliary variable. 
The second order sufficient optimality condition can be used both to verify optimality of candidate solutions as well as to improve the convergence theory of numerical methods such as the discussed Scholtes-type regularization.
Thus, the provided second order results expand the set of optimality conditions for the continuous reformulation of cardinality constrained optimization problems.

Moreover we considered a Scholtes-type regularization to compute S-sta\-tio\-nary points.
Using the previously derived second order conditions we showed two main results:
The existence of local solutions of the regularized programs and a uniqueness result for the limit points.
These extend the existing convergence theory of the Scholtes-type regularization for the continuous reformulation of cardinality constrained optimization problems.
Additionally, we complemented the theoretical results by an example illustrating why the Scholtes-type regularization typically does not get stuck in undesirable local solutions of the continuous reformulation.

\subsubsection*{Acknowledgement}

The work of Alexandra Schwartz and Max Bucher is supported by the 'Excellence Initiative' of the German Federal and State Governments and the Graduate School of Computational Engineering at Technische Universit{\"a}t Darmstadt.
\bibliographystyle{plain}

\begin{thebibliography}{10}

\bibitem{Adam2016}
Luk\'{a}\v{s} Adam and Martin Branda.
\newblock {Nonlinear chance constrained problems: optimality conditions,
  regularization and solvers}.
\newblock {\em Journal of Optimization Theory and Applications},
  170(2):419--436, 2016.

\bibitem{BE2013}
Amir Beck and Yonina~C. Eldar.
\newblock {Sparsity Constrained Nonlinear Optimization: Optimality Conditions
  and Algorithms}.
\newblock {\em SIAM Journal on Optimization}, 23(3):1480--1509, 2013.

\bibitem{BS2009}
Dimitris Bertsimas and Romy Shioda.
\newblock {Algorithm for cardinality-constrained quadratic optimization}.
\newblock {\em Computational Optimization and Applications}, 43(1):1--22, 2009.

\bibitem{B96}
Daniel Bienstock.
\newblock {Computational study of a family of mixed-integer quadratic
  programming problems}.
\newblock {\em Mathematical Programming}, 74(2):121--140, 1996.

\bibitem{BBCS2017}
Martin Branda, Max Bucher, Michal \v{C}ervinka, and Alexandra Schwartz.
\newblock {Convergence of a Scholtes-type Regularization Method for
  Cardinality-Constrained Optimization Problems with an Application in Sparse
  Robust Portfolio Optimization}.
\newblock {\em arXiv preprint arXiv:1703.10637}, 2017.

\bibitem{BKS2015}
Oleg~P. Burdakov, Christian Kanzow, and Alexandra Schwartz.
\newblock {Mathematical Programs with Cardinality Constraints: Reformulation by
  Complementarity-Type Conditions and a Regularization Method}.
\newblock {\em SIAM Journal on Optimization}, 26(1):397--425, 2016.

\bibitem{CW2008}
E.J. Candes and M.B. Wakin.
\newblock {An Introduction To Compressive Sampling}.
\newblock {\em Signal Processing Magazine, IEEE}, 25(2):21--30, March 2008.

\bibitem{CWZ2016}
F.~E. Curtis, A.~W\"{a}chter, and V.~M. Zavala.
\newblock {A Sequential Algorithm for Solving Nonlinear Optimization Problems
  with Chance Constraints}.
\newblock Technical Report 16T-012, COR@L Laboratory, Department of ISE, Lehigh
  University, 2016.

\bibitem{FMPSW2013}
Mingbin Feng, John~E Mitchell, Jong-Shi Pang, Xin Shen, and Andreas
  W\"{a}chter.
\newblock {Complementarity formulations of l0-norm optimization problems}.
\newblock {\em Industrial Engineering and Management Sciences. Technical
  Report. Northwestern University, Evanston, IL, USA}, 2013.

\bibitem{GK2013}
Dinakar Gade and Simge K\"{u}\c{c}\"{u}kyavuz.
\newblock {Formulations for dynamic lot sizing with service levels}.
\newblock {\em Naval Research Logistics (NRL)}, 60(2):87--101, 2013.

\bibitem{Galati2009}
Matthew Galati.
\newblock {Decomposition Methods for Integer Linear Programming}, 2010.
\newblock PhD Thesis.

\bibitem{GLY2013}
Lei Guo, Gui-Hua Lin, and Jane~J. Ye.
\newblock {Second-Order Optimality Conditions for Mathematical Programs with
  Equilibrium Constraints}.
\newblock {\em Journal of Optimization Theory and Applications}, 158(1):33--64,
  2013.

\bibitem{HK2007}
Tim Hoheisel and Christian Kanzow.
\newblock {First-and second-order optimality conditions for mathematical
  programs with vanishing constraints}.
\newblock {\em Applications of Mathematics}, 52(6):495--514, 2007.

\bibitem{LW2014}
Po-Ling Loh and Martin~J Wainwright.
\newblock {Support recovery without incoherence: A case for nonconvex
  regularization}.
\newblock {\em arXiv preprint arXiv:1412.5632}, 2014.

\bibitem{LLRSS2012}
D.~Di Lorenzo, G.~Liuzzi, F.~Rinaldi, F.~Schoen, and M.~Sciandrone.
\newblock {A concave optimization-based approach for sparse portfolio
  selection}.
\newblock {\em Optimization Methods and Software}, 27(6):983--1000, 2012.

\bibitem{LPR1996}
Zhi-Quan Luo, Jong-Shi Pang, and Daniel Ralph.
\newblock {\em {Mathematical programs with equilibrium constraints}}.
\newblock Cambridge Univ. Press, Cambridge, 1996.

\bibitem{Miller2002}
A.~Miller.
\newblock {\em {Subset Selection in Regression}}.
\newblock Chapman \& Hall/CRC Press, 2nd ed. edition, 2002.

\bibitem{MuSh2011}
Walter Murray and Howard Shek.
\newblock {A local relaxation method for the cardinality constrained portfolio
  optimization problem}.
\newblock {\em Computational Optimization and Applications}, 53(3):681--709,
  2012.

\bibitem{NW2006}
Jorge Nocedal and Stephen~J. Wright.
\newblock {\em {Numerical optimization}}.
\newblock {Springer series in operations research and financial engineering}.
  Springer, New York [u.a.], 2. ed. edition, 2006.

\bibitem{OKZ1998}
J.V. Outrata, M.~Ko\v{c}vara, and J.~Zowe.
\newblock {\em {Nonsmooth Approach to Optimization Problems with Equilibrium
  Constraints}}.
\newblock {Nonconvex Optimization and its Applications}. Kluwer Academic
  Publishers, 1998.

\bibitem{PXF2017}
LiLi Pan, NaiHua Xiu, and Jun Fan.
\newblock {Optimality conditions for sparse nonlinear programming}.
\newblock {\em Science China Mathematics}, 60(5):759--776, 2017.

\bibitem{SS2000}
H.~Scheel and S.~Scholtes.
\newblock {Mathematical programs with complementarity constraints:
  Stationarity, optimality, and sensitivity}.
\newblock {\em Mathematics of Operations Research}, 25(1), 2000.

\bibitem{S2001}
Stefan Scholtes.
\newblock {Convergence Properties of a Regularization Scheme for Mathematical
  Programs with Complementarity Constraints}.
\newblock {\em SIAM Journal on Optimization}, 11(4):918--936, 2001.

\bibitem{SU2010}
Sonja Steffensen and Michael Ulbrich.
\newblock {A New Relaxation Scheme for Mathematical Programs with Equilibrium
  Constraints}.
\newblock {\em SIAM Journal on Optimization}, 20(5):2504--2539, 2010.

\bibitem{CKS2016}
Michal \v{C}ervinka, Christian Kanzow, and Alexandra Schwartz.
\newblock {Constraint qualifications and optimality conditions for optimization
  problems with cardinality constraints}.
\newblock {\em Mathematical Programming}, 160(1-2):353--377, 2016.

\bibitem{EWS2003}
Jason Weston, Andr\'{e} Elisseeff, Bernhard Sch\"{o}lkopf, and Pack Kaelbling.
\newblock {The use of zero-norm with linear models and kernel methods}.
\newblock {\em Journal of Machine Learning Research}, pages 1439--1461, 2003.

\bibitem{GY2016}
Ganzhao Yuan and Bernard Ghanem.
\newblock {Sparsity constrained minimization via mathematical programming with
  equilibrium constraints}.
\newblock {\em arXiv preprint arXiv:1608.04430}, 2016.

\end{thebibliography}

\end{document}